\documentclass[11pt,reqno,a4paper]{amsart}
\usepackage{amsmath,amsthm,verbatim,amscd,amssymb,setspace,enumitem,hyperref,esint,mathtools}
\usepackage{exscale,color}

\usepackage{enumitem, MnSymbol}
\usepackage{mathrsfs}
\usepackage{tikz}

\renewcommand{\epsilon}{\varepsilon}

\newcommand{\Z}{\mathbb{Z}}
\newcommand{\Q}{\mathbb{Q}}
\newcommand{\R}{\mathbb{R}}
\newcommand{\C}{\mathbb{C}}
\renewcommand{\P}{\mathbb{P}}

\newcommand{\Cee}{\mathcal{C}}

\newcounter{mtheorem}

\renewcommand{\P}{\mathbb{P}}
\newcommand{{\vol}}{\rm vol}
\newcommand{\p}{\partial}

\newcommand{\D}{\mathcal{D}}
\newcommand{\M}{\mathcal{M}}
\newcommand{\Ric}{\operatorname{Ric}}

\def \C {\mathbb C}
\def \Z {\mathbb Z}
\def \R {\mathbb R}
\def \D {\mathcal D}

\def \M {\mathcal M}
\def \K {\mathcal K}
\def \G {\mathbf G}

\def \P {\mathbb P}

\def \H {\mathbf H}
\def \T {\mathbb T}

\def \p {\partial}
\def \bp {\bar{\partial}}

\def \O {\mathcal{O}}
\def \Id {\text{Id}}

\def \tv {\tilde{v}}
\def \tw {\tilde{w}}

\def \Cstar {\mathbb C^*}
\def \Cstarn {(\mathbb C^*)^n}

\def \p {\partial}
\def \bp {\bar{\partial}}
\def \ddt {\frac{\partial}{\partial t}}
\def \rad {r\frac{\partial}{\partial r}}

\def \Ric {\text{Ric}}
\def \Bl {\text{Bl}}

\def \t {\mathfrak{t}}

\def \Scal {\textnormal{Scal}}
\def \Scalv {\textnormal{Scal}_v}

\def\tr{\operatorname{tr}}
\def\Id{\operatorname{Id}}

\def \Cstarn{(\mathbb{C}^*)^n}

\def \Cstar{\mathbb{C}^*}
\def \t {\mathfrak{t}}
\def \bp {\bar{\partial}}

\def \scaryH {\mathcal{H}_{\alpha, T}^{\varepsilon}}

\def\tr{\operatorname{tr}}
\def\Ric{\operatorname{Ric}}
\def\tr{\operatorname{tr}}

\def\Id{\operatorname{Id}}

\def\vol{\operatorname{vol}}

\newtheoremstyle{fancy}{}{}{\itshape}{}{\textbf\bgroup}{.\egroup}{ }{}
\newtheoremstyle{fancy2}{}{}{\rm}{}{\textbf\bgroup}{.\egroup}{ }{}

\theoremstyle{fancy}
\newtheorem{theorem}{Theorem}[section]
\newtheorem{lemma}[theorem]{Lemma}
\newtheorem{corollary}[theorem]{Corollary}

\newtheorem{prop}[theorem]{Proposition}
\newtheorem{conj}[theorem]{Conjecture}

\theoremstyle{fancy2}
\newtheorem{definition}[theorem]{Definition}
\newtheorem{example}[theorem]{Example}
\newtheorem{remark}[theorem]{Remark}

\setlist{leftmargin=*}

\numberwithin{equation}{section}

\begin{document}
\title{Weighted K-stability for a class of non-compact toric fibrations}
\date{\today}

\author{Charles Cifarelli}
\address{D\'{e}partement de Math\'{e}matiques, Universit\'{e} de Nantes, 2 rue de la Houssini\`{e}re,  BP 92208, 44322 Nantes cedex 03, France}
\email{charles.cifarelli@univ-nantes.fr}

\begin{abstract}
    We study the weighted constant scalar curvature, a modified scalar curvature introduced by Lahdili \cite{LahdiliWeighted} depending on weight functions $(v, \, w)$, on non-compact semisimple principal toric fibrations. The latter notion is a generalization of the Calabi Ansatz originally defined by Apostolov-Calderbank-Gauduchon-T{\o}nnesen-Friedman \cite{ACGT}. This setup turns out to reduce the weighted cscK problem on the total space to a \emph{different} weighted cscK problem on a fixed toric fiber $M$. We show that the natural analog of the weighted Futaki invariant of \cite{LahdiliWeighted} can under reasonable assumptions be interpreted on an unbounded polyhedron $P \subset \R^n$ associated to $M$. In particular, we fix a certain class $\mathcal{W}$ of weights $(v, \,w)$, and prove that if $M$ admits a weighted cscK metric, then $P$ is K-stable, and we give examples of weights on $\C^2$ for which the weighted Futaki invariant vanishes but do not admit $(v,\, w)$-cscK metrics. Following \cite{simonYTD}, we introduce a weighted Mabuchi energy $\M_{v,w}$ and show that the existence of a $(v, \, w)$-cscK metric implies that it $\M_{v,w}$ proper. The well-definedness of $\M_{v,w}$ in this setting also allows us to prove a uniqueness result using the method of \cite{GuanUniqueness}.  As an application, we show that weighted K-stability of the abstract fiber $\C$ is sufficient for the existence of weighted cscK metrics on the total space of line bundles $L \to B$ over a compact K\"ahler base, extending the result in \cite{LahdiliWeighted} in the $\P^1$-bundles case. As a consequence, we recover a well-known existence result for shrinking K\"ahler-Ricci solitons \cite{FIK, FutWang, ChiLiexamples}. Finally, we give some interpretations in terms of asymptotic geometry. 
\end{abstract}

\maketitle 

\section{Introduction}

Let $(B, \,\omega_B)$ be a compact K\"ahler manifold, and suppose that we are given a line bundle $\pi:L \to B$ with $\omega_B \in 2\pi c_1(L)$ and a hermitian metric $h$ on $L$ with curvature equal to $\omega_B$. The \emph{Calabi Ansatz} \cite{Calabi-ansatz} seeks to find a K\"ahler metric on either the total space $Y$ of $L$ or its projectivization $\bar{Y} = \P(\O \oplus L)$, which takes the form
\begin{equation}\label{calabiansatz}
    \tilde{\omega} = \pi^*\omega_B + i\p\bp F(s), 
\end{equation}
where $s = || \cdot ||^2_{h}:L \to \R_+$ is the associated norm function and $F$ is a convex function on $\R_+$. One major advantage of this construction is that many equations on $\omega$ reduce to an ODE for $F$. This has been used with much success to construct constant scalar curvature (cscK) metrics \cite{HS}, K\"ahler-Ricci solitons \cite{FIK, Fut, FutWang, Yang}, and others.   Such a metric $\tilde{\omega}$ always admits a moment map $\mu_{\tilde{\omega}}:Y \to \R$ (resp. $\bar{Y} \to \R$) for the natural $S^1$-action, whose image contains important information about $\tilde{\omega}$. In particular, if one succeeds in constructing a metric on $\bar{Y}$, then the image is necessarily a bounded interval $[a,b]$. In the non-compact case on $Y$ the situation is slightly more subtle, but in many reasonable circumstances a complete metric will give rise to a moment map whose image is an unbounded interval $[a, \infty)$.

A key idea which can be found already in the original work of Calabi \cite{Calabi-ansatz}, but was greatly expanded upon by Apostolov-Calderbank-Gauduchon-T{\o}nnesen-Friedman \cite{ACGT}, is that metrics of the form \eqref{calabiansatz} in fact can be completely described in terms of an associated K\"ahler metric $\omega$ on an abstract fiber $\C$ (resp. $\P^1$), which admits a moment map $\mu: \C \to \R$ (resp. $\P^1 \to \R$) whose image is precisely the same as that of $\tilde{\omega}$. This is done using the fact that we may write 
\[ L = \C \times_{U(1)} U_h \to B, \hspace{.3in} \P(\O \oplus L) = \P^1 \times_{U(1)} U_h \to B, \]
where $U_h$ is the associated $U(1)$-bundle to $(L,\,h)$, and then by interpreting \eqref{calabiansatz} after pulling back to $\C \times U_h$ (resp. $\P^1 \times U_h$).  The upshot is that in many cases if we seek a canonical (cscK, K\"ahler-Ricci soliton, Sasaki-Einstein, etc.) metric $\tilde{\omega}$, then this can be interpreted as a \emph{weighted} cscK problem on the abstract fiber.

For the Calabi Ansatz \eqref{calabiansatz}, this is reflected in the usual derivation of an ODE for $F$. However this perspective allows for a broad generalization, with the help of the theory of $(v,\,w)$-cscK metrics initiated by Lahdili \cite{LahdiliWeighted}, and then developed by Apostolov-Lahdili-Jubert, and others \cite{ApJuLa, simonYTD}. The basic picture is to fix an action of a real torus $\T$ on a complex manifold $M$, and consider those K\"ahler metrics on $M$ which are invariant under the $\T$-action. In reasonable circumstances, such metrics will admit a moment map $\mu_\omega:M \to \R^n$ whose image is controlled by the cohomology class $\omega \in \alpha$. As such, one can fix weight functions $(v,\, w) \in C^\infty(\R^n)$ and study the equation $\Scalv = w$ for K\"ahler metrics $\omega$ in $\alpha$. We call such metrics \emph{$(v, \, w)$-cscK}. Here $\Scalv$, the \emph{$v$-scalar curvature}, is a modified notion of scalar curvature introduced by Lahdili \cite{LahdiliWeighted} which turns out to share many of the properties of the scalar curvature of a compact K\"ahler manifold. This includes an analog of the formal infinite-dimensional moment map picture of Donaldson \cite{DonSym} as well as a K-stability package equipped with energy functionals, geodesics, test configurations, etc. Importantly, for the appropriate choices of weights $(v, \, w)$ the theory recovers many known constructions, including of course the constant scalar curvature metrics, but also extremal metrics, K\"ahler-Ricci solitons, Ricci-flat K\"ahler cones, and others \cite{ApJuLa}. 


The idea of the semisimple principal fibration construction of \cite{ACGT} is to generalize this formulation of the Calabi Ansatz by replacing $U_h$ with a principal $\T$ bundle and the fiber $\C, \P^1$ by a hamiltonian $\T$-manifold $M$ as above. More precisely, suppose as before we have a compact K\"ahler base $(B,\, \omega_B)$ and a K\"ahler manifold  $(M,\, \omega_M)$ of dimension $n$ with a hamiltonian action of an $k$-dimensional torus $\T$ with momentum image $P \subset \t^* = \textnormal{Lie}(\T)^*$. Suppose further that we are given a principal $\T$-bundle $\pi_B: U \to B$ whose topology is determined by the cohomology class $[\omega_B]$ in an explicit way (see Section \ref{section-ssfibrations} for details). Then we can consider the \emph{semisimple principal fibration}, which is the total space of the associated bundle
\[ Y = M \times_\T U , \hspace{.3in} \pi_B:Y \to B. \]
Then $Y$ comes equipped with a natural complex structure and holomorphic $\T$-action. Moreover, if $\mu_M: M \to P$ is the moment map with respect to $\omega_X$, we can associate a corresponding $\T$-invariant K\"ahler metric $\omega_Y$ on $Y$. When pulled back to $M \times U$ by the quotient map $\pi_Y: M \times U \to Y$, this has the form
\begin{equation*}
\begin{split}
    \omega_Y &= \omega_M +  \pi_B^*\tilde{\omega}_{B} + d(\langle \mu_M, \theta \rangle),
\end{split}
\end{equation*}
where $\theta \in \Omega^1(U, \t)$ is a particular choice of connection $1$-form and $\tilde{\omega}_B$ is a (potentially different) K\"ahler metric on $B$. For example in the simplest cases, we will have $d\theta = \pi_B^{*}\omega_B \otimes p$ for $p \in \Z^k\subset \t$, and $\tilde{\omega}_B = c \omega_B$ for some constant $c$. Crucially, any such $\omega_Y$ admits a moment map for the $\T$-action $\mu_Y:Y \to \t^*$ whose image is also equal to $P$. The key observation of \cite{LahdiliWeighted, ApJuLa} is that if the scalar curvature of $\omega_B$ itself is constant, then the metric $\omega_Y$ is $(v, \, w)$-cscK if and only if $\omega_M$ is $(\tilde{v}, \, \tilde{w})$-cscK for \emph{different} weights $(\tilde{v}, \, \tilde{w}) \in C^\infty(\overline{P})$ (Lemma \ref{fibrationweights}). 

Therefore, just as for the usual Calabi Ansatz, one can attempt to study the existence problem for weighted constant scalar curvature metrics on $Y$ indirectly by studying a different existence problem on the fiber $M$. The perspective that we take in this paper is to restrict attention to the situation where $M$ is toric and our torus $\T$ satisfies $\dim_\R \T = \dim_\C M$. In this case we say that $Y$ is a \emph{semisimple principal toric fibration}. When $M$ is compact, this approach has already been rather fruitful, see \cite{LahdiliWeighted} and in particular \cite{simonYTD}, where a version of the Yau-Tian-Donaldson conjecture is established for semisimple principal toric fibrations with compact toric fiber $M$.

\subsection{Results}

In this paper, we study the weighted cscK problem on quasiprojective toric manifolds $M$ endowed with a special type of K\"ahler metric that we call an \emph{AK metric} $\omega$. This, by definition, is a K\"ahler metric on $M$ with respect to which the $\T$-action is hamiltonian, and which admits a proper moment map 
\[ \mu: M \to P \subset \t^* = \textnormal{Lie}(\T)^*,\]
where $P$ is now an \emph{unbounded} polyhedron. Crucially, the image $P$ of an AK metric depends only (up to translation in $\t^*$) on the cohomology class $\omega \in \alpha$ \cite{uniqueness}. A new feature in this context is that many results depend intimately on the behavior of the weights $(v, \, w)$ at infinity. For example, consider the \emph{weighted Futaki invariant} of \cite{LahdiliWeighted},  given by 
\[ \mathcal{F}_{v, w}(\ell) = 2 \int_{\p P} \ell v d\sigma - \int_P \ell w dx,\]
where $\ell$ is an affine-linear function on $P$. In particular, we see already that in order to establish a reasonable theory we must have that $v$ and $w$ satisfy some kind of integrability condition.

To begin to tackle the problem, we introduce a specific class $\mathcal{W} = \mathcal{W}(P)$ of weights $(v, \, w)$, which decay exponentially on $P$, i.e. 
\[ v = O(e^{-c_v|x|}), \hspace{.3in} w = O(e^{-c_w|x|}), \hspace{.3in} \textnormal{as } |x| \to \infty. \] 
See Definition \ref{exponential-decay} for the precise condition. One motivation for this type of restriction is the case of \emph{K\"ahler-Ricci solitons}. As we will see in Section \ref{section-KRS} there are in fact several distinct ways to interpret K\"ahler-Ricci solitons as weighted cscK metrics. In all cases, the weights have the form $v(x)= p(x)e^{-\langle x, \, b \rangle}, w(x)= q(x)e^{-\langle x, \, b \rangle}$, for $p, q$ rational functions $p, q$ on $P$ and $b \in \textnormal{Lie}(\T)$.

Given any $v \in \mathcal{W}$, we introduce two important families of associated spaces. We begin by defining a space $\scaryH = \scaryH(v)$ of K\"ahler metrics which are sufficiently controlled by $v$, at a rate parameterized by $\varepsilon \in [0, \frac{1}{2})$. In particular, for a K\"ahler metric $\omega \in \scaryH$ we will always have 
\[ \int_M \Scalv(\omega) \, \omega^n < \infty.\]
Following this, for suitable choices of $\beta > 0$ we can define a space $\Cee_\beta$ of convex functions on $P$, smooth on the interior, and which do not grow too quickly at infinity relative to $v$. This latter space is a direct analog of the function spaces considered in \cite{DonStabTor, ChenLiSheng2014, LiLianSheng1, simonYTD} and many others, adapted to the unbounded weighted setting. Once again we emphasize that a notable distinction here is that the weight $v$ plays a crucial role in controlling the behavior at infinity.

Given this setup, we can define the Futaki invariant $\mathcal{F}_{v,w}$ on all suitable $(v,\,w)$-integrable functions on $P$. We say that $P$ is \emph{K-semistable} if 
\begin{equation*}
    \mathcal{F}_{v,w}(f) \geq 0 
\end{equation*}
for every convex piecewise-linear function $f \in C^0(\overline{P})$. We say that $P$ is \emph{K-stable} if it is K-semistable and $ \mathcal{F}_{v,w}(f) = 0$ if and only if $f$ is affine-linear. Then we have the following generalization of a result of Zhou-Zhu \cite{ZhouZhu}:
\begin{theorem}\label{Mtheorem-existenceimpliesstable}
 Let $M$ be the quasiprojective toric variety associated to the Delzant polyhedron $P \subset \t^*$, and suppose that $(v, \, w)$ are weights in the class $\mathcal{W}(P)$. Then if $M$ admits a $(v,\,w)$-cscK metric $\omega$ with $\omega \in \scaryH$, then $P$ is $(v, \, w)$ K-stable.
\end{theorem}
In Section \ref{section-nonexistence} we give some examples of weights $(v, \, w)$ on $\C^2$ such that the obstruction $\mathcal{F}_{v,w}$ vanishes on the affine-linear functions, but nonetheless we can use Theorem \ref{Mtheorem-existenceimpliesstable} to rule out the existence of $(v, \, w)$-cscK metrics $\omega \in \scaryH$ on $\C^2$ for any $\varepsilon >0$. Moreover, in Section \ref{section-testconfig} we return to the classical Calabi Ansatz and show that weighted K-stability is a sufficient criterion for the existence of certain weighted cscK metrics on line bundles over a compact cscK manifold $(B, \, \omega_B)$. This is a direct generalization of the $\P^1$-bundles case considered in \cite{LahdiliWeighted}, and can be stated as follows:
\begin{prop}[{Proposition \ref{Kstableimpliesexistence-cbundles}}]\label{mainprop}
    Let $P = [-1, \infty)$, and $(\tilde{v}, \, \tilde{w})$ be any weights on $P$ with $\tilde{v} > 0$. Then there exists a $(\tilde{v}, \, \tilde{w})$-cscK metric $\omega$ on $\C$ with moment image equal to $P$ if $P$ is $(\tilde{v}, \, \tilde{w})$ K-stable.  
    \end{prop}
   As a corollary, we give a simple proof in the shrinking case of the following existence result of Futaki-Wang \cite{FutWang} generalizing work of Feldman-Ilmanen-Knopf \cite{FIK} (see also \cite{ChiLiexamples}):
   \begin{corollary}[{\cite{FIK, FutWang, ChiLiexamples}}]\label{Mcorollary2}
   	Let $B$ be a K\"ahler-Einstein Fano manifold, and let $L \to B$ be a negative line bundle with $L^{\frac{1}{\kappa}} = K_B$ for $0 < \kappa < 1 \in \mathbb{Q}$. Then there exists a shrinking K\"ahler-Ricci soliton on the total space of $L$. 
   \end{corollary}
 Moreover, under certain assumptions on the weights which are relevant in the context of shrinking K\"ahler-Ricci solitons (c.f. Lemmas \ref{fibrationsolitons}, \ref{vsolitonw}), we show that the metrics of Proposition \ref{mainprop} indeed have the type of asymptotic behavior that we consider in this paper.
\begin{prop}[{Proposition \ref{mainprop2-inthebody}}]\label{mainprop2}
        In the same situation as Proposition \ref{mainprop}, suppose that the Futaki invariant $\mathcal{F}_{\tilde{v}, \tilde{w}}$ vanishes on the affine-linear functions and that $(\tilde{v}, \, \tilde{w})$ satisfy 
     \[ \tilde{v}(x) = p(x)e^{-\lambda x}, \hspace{.3in} \tilde{w}(x) = q(x) e^{-\lambda x}, \]
        where $q(x)$ is a positive rational function on $P$, $q(x)$ is a polynomial, and $\lambda > 0$. Then the metric $\omega$ on $\C$ given by Proposition \ref{mainprop} satisfies $\omega \in \scaryH$ for any $\varepsilon > 0$.
\end{prop}
As we saw above, the existence of $(\tilde{v}, \, \tilde{w})$-cscK metrics on $\C$ is intimately related to weighted cscK metrics on the total space of certain line bundles $\pi: L \to B$ over compact cscK bases $(B, \, \omega_B)$ with respect to new weights $(v, \, w)$. These new weights can be computed explicitly in terms of the topological data of $L$, and we explore this relationship in Section \ref{section-testconfig}. Moreover, just as in \cite{LahdiliWeighted}, it turns out that we only need to consider a simple special class of piecewise-linear $f$, which in some sense have a uniformly controlled slope. This turns out to have an interesting interplay with \emph{uniform stability}. 

In order to fix our function spaces, we must control for the ambiguity of $P$ in $\t^*$ up to translation, as any two choices of polyhedron in the same translation class will give rise to the same polarized variety (see Lemma \ref{makeavariety}). Thus we normalize by choosing some fixed a representative containing $0 \in \t^*$. Given any $f \in \Cee_\beta$, we can therefore always ensure up to the addition of an affine function that $f$ attains a minimum value of zero at the origin $0 \in P$. We denote by $\Cee_\beta^*$ the set of thus normalized functions. For any $K> 0$, in Section \ref{section-uniform} we define a space $\Cee^*_\beta(K) \subset \Cee_\beta^*$, where we ask for a uniform control on $f$ in an asymptotic $L^1$ sense relative to $v$ (see \eqref{ceekbeta}). We then introduce a notion of uniform stability depending on an extra parameter $\gamma$, by demanding that 
\begin{equation}\label{uniform-intro}
    \mathcal{F}_{v,w}(f) \geq \lambda \int_{P} f v^{\gamma} dx
\end{equation} 
for all $f \in \Cee_\beta^*(K)$. In particular, we say that $P$ is \emph{$\beta$-uniformly K-stable} if \eqref{uniform-intro} holds for $\gamma = \beta$. 

We also define a weighted Mabuchi energy $\mathcal{M}_{v,w}$ on the whole of $\Cee_\beta$, following the constructions of \cite{LahdiliWeighted, simonYTD}. It is interesting to note that in the non-compact case the usual (unweighted) Mabuchi energy is not in general well-defined unless we restrict attention to metrics which are very close to a fixed model at infinity. As we will see, in order to make sense of the weighted functional $\M_{v,w}$, the strong decay of the weights at infinity allows us to relax this condition significantly. Summarizing the remaining main results of Section \ref{stability-section}, we have the following generalization of \cite{ChenLiSheng2014, simonYTD}:
\begin{theorem}\label{Mtheorem-uniform}
    Let $M$ be the quasiprojective toric variety associated to the Delzant polyhedron $P \subset \t^*$, and suppose that $(v, \, w)$ are weights in the class $\mathcal{W}$. Then we have 
    \begin{enumerate}
        \item if $M$ admits a $(v, \, w)$-cscK metric $\omega$ in $\scaryH$, then $P$ is $\beta$-uniformly  K-stable for all suitable $\beta$,
        \item if $P$ is $\beta$-uniformly K-stable, then the weighted Mabuchi energy is proper on $\Cee_\beta^*(K)$.
    \end{enumerate}
\end{theorem}
Using the strategy of \cite{GuanUniqueness}, the convexity of the weighted Mabuchi functional allows us to obtain the following basic uniqueness result: 
\begin{corollary}\label{Mcorollary}
    Let $M, P, (v,\, w)$ be as in Theorem \ref{Mtheorem-uniform}. For any suitable $\varepsilon \in \left(0 , \frac{1}{2}\right)$ (see Definition \ref{asymptotic-metric-general}), suppose there exists a $(v, \, w)$ cscK metric $\omega \in \scaryH$. Then if $\omega' \in \scaryH$ is another $(v, \, w)$-cscK metric which is uniformly equivalent to $\omega$, then there exists an automorphism $A:M \to M$ such that $A^*\omega' = \omega$.
\end{corollary}
The class of metrics $\scaryH$ that we consider is fairly broad. To give some geometric significance, we describe some explicit criteria on a K\"ahler metric $\omega$ in some special cases which ensure that $\omega \in \scaryH$. To this end, suppose that $V$ is a smooth projective toric variety and that $M_c$ is an affine toric variety whose singular set is equal either to the unique torus fixed point $\{o\}$ or is empty (in which case $M_c \cong \C^n$). Let $M$ be a smooth quasiprojective toric variety with $\pi: M \to V \times M_c$ a torus-equivariant proper birational morphism, which is an isomorphism away from the set $V \times \{o\} \subset V \times M_c$. We say that a K\"ahler metric $g$ on $M$ is \emph{asymptotically c-cylindrical} if there exists a K\"ahler metric $g_V$ on $V$ and a conical metric $g_c$ on $M_c$ with radial function $r = r(p) =\textnormal{dist}_{g_c}(p,\{o\})$ and 
\begin{equation}\label{accyl1}
		\left| \pi_*g- g_0 \right|_{g_0} < Cr^{- 2}, \hspace{.3in} \left|\nabla^{g_0}( \pi_*g- g_0) \right|_{g_0} < Cr^{-1}.
\end{equation}
In order to make a connection with the other ideas in this paper, we also require a technical condition on the symplectic structure associated to $(g, J)$ on $M$, see Definition \ref{asymptoticsproduct}.
\begin{theorem}\label{mtheorem-asymptotics}
    Let $\pi: M \to V \times M_c$ be a resolution of $V \times M_c$ as described above. Let $P \subset \t^*$ be a Delzant polyhedron corresponding to the K\"ahler class $\alpha \in \K_M$, and suppose that $(v, \, w)$ are weights in the class $\mathcal{W}$. Then any asymptotically c-cylindrical metric $\omega$ on $M$ with $\omega \in \alpha$ lies in $\scaryH$ for any $\varepsilon > 0$. Consequently if $M$ admits an asymptotically c-cylindrical $(v, \, w)$-cscK metric $\omega \in \alpha$, then $P$ is $(v, \, w)$ K-stable and $\beta$-uniformly  K-stable for all suitable $\beta$.
\end{theorem}
In Section \ref{section-ssfibrations} we also take an interest in the case of shrinking K\"ahler-Ricci solitons. We explain how these metrics give examples of the type of weighted problem we consider in the case of semisimple principal toric fibrations. Returning to the question of asymptotics, we study the behavior at infinity of shrinking K\"ahler-Ricci solitons on the total space of certain negative line bundles $L \to B$ over a compact K\"ahler-Einstein Fano base,  generalizing the examples obtained via the Calabi Ansatz of Feldman-Ilmanen-Knopf \cite{FIK} on negative line bundles $\O(-k) \to \P^{n-1}$.  To the author's knowledge, the type of examples we consider here were first shown to exist in \cite{Yang}.  For another similar approach in the case that the base $B$ is not necessarily K\"ahler-Einstein but instead toric, see \cite{FutWang, Fut}.  We follow the presentation of Li \cite{ChiLiexamples}, who in fact finds K\"ahler-Ricci solitons on direct sum bundles $L \oplus \dots \oplus L \to B$. Specifically, we use the relationship to the semisimple principal fibration construction mentioned above to see that the metrics on $L$ come from weighted cscK-metrics on $\C$. We use the asymptotic model identified in \cite{ChiLiexamples} (which is a cone metric) to show that in fact these weighted metrics on $\C$ satisfy the asymptotic condition \eqref{accyl1}.

\subsection*{Acknowledgements}
I would like to begin by thanking Simon Jubert for our conversations while he was visiting the Universit\'{e} de Nantes during the Fall of 2022, which served as the inspiration for this article. I would like to thank Vestislav Apostolov for his kind guidance and many useful suggestions, and Song Sun for comments on preliminary versions of the article. I would also like to thank the referee for helpful comments and suggestions. Finally, I would like to thank Paula Burkhardt-Guim and David Tewodrose for useful discussions. The author is supported by the grant Connect Talent ``COCOSYM'' of the r\'{e}gion des Pays de la Loire and the Centre Henri Lebesgue,  programme ANR-11-LABX-0020-0. 

\section{Background}

\subsection{Polyhedra and toric varieties}

In this section we introduce some important classes of combinatorial objects that we will use throughout the paper.  Fix a real torus $\T = T^n$ of dimension $n$ and denote its Lie algebra by $\t$, and corresponding dual $\t^*$. Let $\Gamma \cong \Z^n \subset \t$ be the lattice defined by $b \in \Gamma \iff \exp(b) = \Id \in \T$.

\begin{definition}\label{polydefdef}
	A \emph{polyhedron} $P \subset \t^*$ is any convex finite intersection of affine half spaces $H_{\nu,a} =\{x \in \t^* \: | \:  \langle \nu, x \rangle \geq -a \}$ with $\nu \in \t,\, a \in \R$. A \emph{polytope} is a bounded polyhedron.  If $P$ has at least one vertex and all of the vertices of $P$ lie in the dual lattice $\Gamma^* \subset \t^*$, then we say that $P$ is \emph{rational}.
\end{definition}

We set $L_{\nu}(x) = \langle \nu, \, x \rangle$, so that the half space $H_{\nu, a}$ is defined by the linear inequality $L_\nu(x) \geq -a$. As such, we can write 
\begin{equation}\label{polydef}
 P = \left\{x \in \t^* \: | \:  L_{\nu_i}(x) + a_i \geq 0 , \, i = 1, \dots, N   \right\}.
\end{equation}
Given such a presentation, we refer to the collection $\nu_i \in \t$ as the \emph{inner normals} of $P$. For the purposes of this paper, we will always assume that a polyhedron $P$ has open interior. We will often not distinguish between a polyhedron $P$ and its interior, but where confusion may arise we will denote by $\overline{P}$ the closed object and $P$ the interior. The intersection of $P$ with the plane $L_{\nu} = -a$ is a polyhedron $F_\nu$ of one less dimension and is called a \emph{facet} of $P$. The intersections of any number of the facets $F_\nu$ form the collection of \emph{faces} of $P$. 

\begin{definition} \label{recessioncone}
Let $P$ be a polyhedron given by the intersection of the half spaces $H_{\nu_i, a_i}$. We define the \emph{recession cone} (or asymptotic cone) $C(P)$ by 

\begin{equation}\label{rpcone}
	C(P) = \left\{ x \in \t^* \: | \: L_{\nu_i}(x) \geq 0 \right\}.
\end{equation}
In general any polyhedron of the form \eqref{rpcone} with $\nu_i \in \Gamma$ is called a \emph{rational polyhedral cone}. 
\end{definition}
\noindent Given any convex cone $C \subset \t$, the \emph{dual cone} $C^* \subset \mathfrak{t}$ is defined by

\begin{equation} \label{dualcone}
	C^* = \{\xi \in \mathfrak{t} \: | \: \langle \xi, x \rangle \geq 0 \text{ for all } x \in C\}.
\end{equation}
Note that (the interior of) the dual recession cone $C(P)^*$ to a polyhedron $P$ is necessarily an open cone in $\mathfrak{t}$, even when $C(P)$ is not full-dimensional.

\begin{lemma}[{\cite[Theorem 7.1.10]{CLS}}]\label{makeavariety}
To any full-dimensional rational polyhedron $P$ in $\t^*$, there exists an associated quasiprojective variety $M_P$ with a regular action of the algebraic torus $\Cstarn$ and a unique orbit which is open and dense.  Moreover, if we set $\t_C^* \subset \t^*$ to be the smallest linear subspace which contains the recession cone $C(P)$, then there exists an affine variety $M_C$ associated to $C := C(P) \subset  \t_C^*$ in the same fashion, and there is a natural projective morphism $\pi_C: M_P \to M_C$. 
\end{lemma}

Moreover, it is essentially true that any quasiprojective toric variety $M$ is of the form $M = M_P$ for a rational polyhedron $P$, see \cite[Proposition 7.2.9]{CLS}.

\begin{definition} \label{Delzant}
	Let $P$ be a full-dimensional polyhedron in $\t^*$ with at least one vertex. Then $P$ is called \emph{Delzant} if, for each vertex $v \in P$, there are exactly $n$ edges $e_i$ stemming from $p$ which can be written $e_i = v + \lambda_i \varepsilon_i$ for $\lambda_i \in \R$ and $(\varepsilon_i)$ a $\Z$-basis of $\Gamma^*$. 
\end{definition}

This says that each vertex of a Delzant polyhedron, when translated to the origin, can be made to look locally like standard $\R^n_+$ via an element of $\text{GL}(n,\Z)$.  Hence we see that the inner normals $\nu_i$ to a Delzant polyhedron lie in $\Gamma$. 

\begin{lemma}[{\cite[Theorem 3.1.19]{CLS}}]\label{makeamanifold}
	The variety $M_P$ of Lemma \ref{makeavariety} is smooth if and only if $P$ is Delzant. 
\end{lemma}

Note that the Delzant condition \eqref{Delzant} is simpler than what one sometimes is forced to consider in other non-compact situations (see for example \cite{Ler, KarLer}). The difference is that in our setting, $M_P$ can be covered by holomorphic charts centered at the fixed points. In particular, if there is only one vertex of $P$, then $M_P \cong \C^n$. 

\begin{definition}\label{noncompacttoric}
 We say that a non-compact complex manifold $M$ is \emph{toric} if it is biholomorphic to the toric variety $M_P$ associated to a Delzant polyhedron $P \subset \t^*$. 
\end{definition}

Consequently, any toric manifold $M$ in this context satisfies $H^1(M) = 0$ and admits an effective and holomorphic $\Cstarn$-action with finite and nonempty fixed point set.

\newpage 

\subsection{K\"ahler metrics}

\begin{definition}\label{akmetric}
	Fix a K\"ahler class $\alpha \in \K_M \subset H^{1,1}(M)$. An \emph{AK metric} is then a K\"ahler metric $\omega \in \alpha$ which admits at least one hamiltonian potential $h: M \to \R$ which is proper and bounded below.  Fixing a base $\omega \in \alpha$, we set $\mathcal{H}_{\alpha,T} \subset C^\infty(M)$ to be the set of $\varphi$ such that $\omega_\varphi = \omega_0 + i\p\bp \varphi$ is an AK metric. 
\end{definition}
Note that clearly $\mathcal{H}_{\alpha,T}$ is convex as a subset of $C^\infty(M)$. 
\begin{lemma}[{\cite[Section 2.5]{uniqueness}}]\label{goodimage}
    Any $\omega \in \mathcal{H}_{\alpha,T}$ admits a moment map $\mu_\omega:M \to \t^*$ whose image is a Delzant polyhedron $P$, determined uniquely up to translation by $\alpha$. Moreover, $M$ is biholomorphic to $M_P$.
 \end{lemma}
We can always choose a representative polyhedron $P$ such that the origin $0 \in \t^*$ is contained in the interior.  The restriction of a $\T$-invariant K\"ahler metric on $M$ to the dense orbit $\Cstarn \subset M$ has a characterization due to Guillemin. 

\begin{prop}[{\cite[Theorem 4.1]{Guil}}]\label{guillemin-exact} 
	Let $\omega$ be any $\T$-invariant K\"ahler form on $\Cstarn$. Then the action is hamiltonian with respect to $\omega$ if and only if there exists a $\T$-invariant potential $\phi$ such that $\omega = 2i\p\bp \phi$. 
\end{prop}

Hence, any AK metric on $M$ can be written $\omega= 2i\p\bp \phi$ on the dense orbit.  We fix once and for all such a basis $(Y_1, \dots, Y_n)$ for $\t$. This induces a background coordinate system $(\xi^1, \dots, \xi^n)$ on $\t$. We use the natural inner product on $\t$ to identify $\t \cong \t^*$ and thus can also identify $\t^* \cong \R^n$. For clarity, we will denote the induced coordinates on $\t^*$ by $(x^1, \dots, x^n)$. Let $(z_1, \dots, z_n)$ be the natural coordinates on $\Cstarn$ as an open subset of $\C^n$. There is a natural diffeomorphism $\text{Log}: \Cstarn \to \t \times \T$, which provides a one-to-one correspondence between $\T$-invariant smooth functions on $\Cstarn$ and smooth functions on $\t$. Explicitly, $\text{Log}(z_1, \dots, z_n) = (\log(r_1), \dots, \log(r_n), \theta_1, \dots, \theta_n)$, where $z_j = r_j e^{i \theta_j}$. Given a function $F(\xi)$ on $\t$, we can extend $F$ trivially to $\t \times \T$ and pull back by Log to obtain a $\T$-invariant function on $\Cstarn$. Clearly, any $\T$-invariant function on $\Cstarn$ can be written in this form.

In particular, we have that $\phi$ itself is determined by a smooth function on $\t$, and so we henceforth write $\phi = \phi(\xi)$.  Writing $w_i = \log(z_i)$, we see that

\begin{equation*}
	\omega = 2i\frac{\p^2 \phi}{ \p w^i \p\bar{w}^j} dw_i \wedge d\bar{w}_j = \frac{\p^2 \phi}{ \p \xi^i \p\xi^j} d\xi^i \wedge d\theta^j, 
\end{equation*}
and moreover the metric $g$ corresponding to $\omega$ is given on $\t \times \T$ by 

\begin{equation}\label{metricexpression-complexcoord}
	g = \phi_{ij}(\xi)d\xi^i d\xi^j + \phi_{ij}(\xi)d\theta^i d\theta^j.
\end{equation}
Since the hessian of $\phi$ is thus positive-definite, it follows that $\phi$ is strictly convex on $\t$, so that in particular $\nabla \phi$ is a diffeomorphism onto its image. In fact:

\begin{lemma} \label{propernesslemma}
	Let $\phi$ be any smooth and strictly convex function on $\R^n$. Let $\Omega = \nabla\phi(\R^n)$. Then if $0 \in \Omega$, there exists a $C  >0$ such that 	
	\begin{equation} \label{properness}
		\phi(\xi) \geq C^{-1}|\xi| - C.
	\end{equation}
\end{lemma}

Considering that $\nabla \phi: \t \to \t^*$ is a diffeomorphism onto its image, we can therefore use it to change coordinates and study the geometry of the dense orbit $\Cstarn \subset M$ directly on $\Omega = \nabla \phi(\t) \subset \t^*$. Under this change, the geometry is now encoded by the \emph{Legendre transform} $u$ of $\phi$:

\begin{equation} \label{legendretransdef}
	\phi(\xi) + u(x) =  \langle \xi, x \rangle,
\end{equation}
where $x = \nabla \phi(\xi)$. The metric $g$ is given on $\Omega$ by 
\begin{equation*}
	g = u_{ij}(x)dx^i dx^j + u^{ij}(x)d\theta_i d\theta_j.
\end{equation*} 
Thus the metric structure is determined by the hessian of the function $u$, and so by analogy with the complex case this function is sometimes called the \emph{symplectic potential} for $g$. 

\begin{lemma} \label{legendreproperties}
	Let $\t$ be a real vector space and $\phi$ be a smooth and strictly convex function on $\t$. Then there is a unique function $L(\phi) = u$ defined on $\Omega = \nabla\phi(\Omega') \subset \t^*$ by \eqref{legendretransdef}:
	\begin{equation*} 
		\phi(\xi) + u(x) = \langle \xi, x \rangle
        \end{equation*}
 for $x = \nabla \phi (\xi)$. The function $u$ is smooth and strictly convex on $\Omega$. Moreover, $L$ has the following properties:

\begin{enumerate}
	\item $L(L(\phi)) = \phi $,
	\item $\nabla \phi: \Omega' \to \Omega$ and $\nabla u: \Omega \to \Omega'$ are inverse to each other,
	\item $\textnormal{Hess}_{\xi}\, \phi(\nabla u(x)) = \textnormal{Hess}\, u^{-1}(x) $,
	\item $L((1-t)\phi + t \phi') \leq (1-t)L(\phi) + t L(\phi')$.
\end{enumerate}

\end{lemma}

It's clear that the Euclidean gradient $\nabla \phi: \t \times \T \to \t^*$ satisfies 
	\begin{equation*}
		d  \langle \nabla \phi(\xi), b \rangle = - i_{Y_b} \omega
	\end{equation*}
for all $b \in \t$ where $Y_b = b^i \frac{\p}{\p \theta^i}$. Thus, either by adding a constant in $\t$ to $\mu$ or by adding a linear function to $\phi$, we can always ensure that 
\begin{equation}\label{normalizedgradientmomentmap}
    \mu_\omega|_{\Cstarn} = \nabla \phi.
\end{equation} 
Under this choice of normalization, any further change to $\phi$ given by the addition of a linear function will have the effect of translating the moment map $\mu$. If we assume that $\omega$ is an AK metric, then by Lemma \ref{goodimage} $\nabla \phi(\t)$ is equal to a Delzant polyhedron $P$ and $(M,J)$ is biholomorphic to $M_P$. In this setting, there is a natural K\"ahler metric on $M$ \cite{Del, Guil, BGL} whose symplectic potential is given by 
\begin{equation}\label{guilleminpotential}
    u_P = \frac{1}{2} \sum_{i = 1}^N (L_{\nu_i}(x) + a_i) \log(L_{\nu_i}(x) + a_i),
\end{equation}
where $P$ is presented as in \eqref{polydef}. 

For any AK metric $\omega$, let $\mathbf{H} = (H_{ij}) \in C^\infty(\t^* \otimes \t^*)$ be defined by 
    \begin{equation}\label{Hdef}
		H_{ij} = g(Y_i, Y_j),
	\end{equation} 
for a basis $Y_1, \dots, Y_n$ of $\t$. From \eqref{metricexpression-complexcoord} and Lemma \ref{legendreproperties} we see that at any point in the dense orbit we have that $\H = \textnormal{Hess}^{-1}(u)$. Then $\mathbf{G} = (G_{ij})$ is defined to be
	\begin{equation}\label{Gdef}
		\G = \H^{-1},
	\end{equation} 
so that $\G = \textnormal{Hess}(u)$ on the dense orbit. 
    
In \cite{uniqueness}, it was observed that the boundary conditions classically known to be satisfied by the symplectic potential in the compact setting \cite{ACGT} also hold here:

\begin{prop}[{\cite[Proposition 2.17]{uniqueness}, \cite[Proposition 1]{ACGT}}]\label{boundaryconditions}
	Let $\mu_\omega: M \to \t^*$ be the moment map associated to an AK metric $\omega \in \alpha$, whose image is equal to Delzant polyhedron $P$ by Lemma \ref{goodimage}. Then the symplectic potential $u_\omega$ associated to $\omega$ defined via the Legendre transform as above satisfies $u_\omega - u_P \in C^\infty(\overline{P})$, where $u_P$ denotes the Guillemin potential \eqref{guilleminpotential}. Moreover, the data $\H$ given by \eqref{Hdef} satisfies: 
    \begin{itemize}
        \item $\H$ is the restriction to $P$ of a smooth $C^\infty(\t^* \otimes \t^*)$-valued function on $\overline{P}$.
        \item For any point $y$ contained in a facet $F$ of $P$ with inner normal $\nu_F$, we have 
         \[ \H_y(\nu_F, \, \cdot ) = 0,  \hspace{.3in} d\H_y(\nu_F, \nu_F) = 2 \nu_F. \] 
         \item for any point $y$ contained in a face $F$ (of any codimension), the restriction of $\H_y(\cdot, \, \cdot)$ to $(\t/\t_F)^*$ is positive-definite. 
    \end{itemize}
\end{prop}

\begin{remark}
 One can also consider a more general framework, where we allow arbitrary $C^\infty(\t^* \otimes \t^*)$-valued functions $\H$ satisfying the conditions above which are not necessarily given by the inverse hessian of a function. This will then give rise to \emph{almost} K\"ahler structures on $M$, i.e. to almost complex structures $J_{\H}$ on $M$ compatible with a fixed symplectic form $\omega$ but which are not necessarily integrable. This was in fact the original generality treated in \cite[Proposition 1]{ACGT} (see also \cite{Lejmi-almost, Legendre-almost}).  Many of the results here could surely be extended to the almost K\"ahler setting. In this paper, however, we will focus on the integrable case, so that $\H$ is indeed equal to the inverse hessian of a function.
\end{remark}

\begin{lemma}\label{gradexpressionlemma}
Let $M$ be toric with AK metric $\omega$ and moment image $P \subset \t^*$. Let $F$ be a smooth function on the interior of $P$. Then we have that for any $p$ lying in the dense orbit, we have 
\begin{equation}\label{gradexpression}
	| \nabla^{g}F(\mu_{\omega})|^2_{g}(p) = H_{kl}\frac{\p F}{\p x^i} \frac{\p F}{\p x^j}. 
\end{equation}
\end{lemma}
\begin{proof}
As usual, write $\omega = 2i\p\bp \phi$ on the dense orbit with $\phi$ normalized such that $\mu_\omega = \nabla \phi$.  We calculate
\begin{equation*}
 	| \nabla^{g}F(\mu_{\omega})|^2_{g}(p) = \phi_{kl} \phi^{il} \phi^{kj} \frac{\p }{\p \xi^i} F(\nabla \phi) \frac{\p }{\p \xi^j} F(\nabla \phi). 
\end{equation*}
Using the properties of the Legendre transform (Lemma \ref{legendreproperties}), we see that 
\begin{equation*}
	\phi^{il} \frac{\p }{\p \xi^i} = \frac{\p }{\p x^l},
\end{equation*}
which consequently gives \eqref{gradexpression}. 

\end{proof}

\subsection{The $v$-scalar curvature}\label{section-vscal}

The principal object of study in this paper will be the weighted scalar curvature, introduced by Lahdili in \cite{LahdiliWeighted}. Suppose that $M$ is an arbitrary complex manifold with an effective holomorphic $\T$-action, and which admits a K\"ahler metric $\omega$ with respect to which this action is hamiltonian. Denote as usual by $P$ the image $\mu_\omega(M) \subset \t^*$, but note for the moment with these minimal assumptions that $P$ need not have any special structure. Let $v \in C^\infty(\overline{P}, \R_{>0})$. Then we define the $v$-scalar curvature $\Scalv(\omega)$ of $\omega$ by 
\begin{equation}\label{vscal}
	\Scalv(\omega) = v(\mu_\omega) \Scal(\omega) + 2 \Delta_{\omega} v(\mu_{\omega}) +\langle g, \,  \mu_{\omega}^*\textnormal{Hess}(v) \rangle ,
\end{equation}
where $\langle \cdot, \, \cdot \rangle$ denotes the dual pairing between $\t \otimes \t$ and $\t^* \otimes \t^*$. Then the main object of interest for the purposes of this paper are solutions $\omega$ of the \emph{$(v,\,w)$-cscK equation}:
\begin{equation}\label{vwcscK}
    \Scalv(\omega) = w,
\end{equation}
for prescribed $w \in C^\infty(\overline{P})$. 

If $M$ is in addition toric, then we can understand the $v$-scalar curvature more explicitly. Indeed, the local computations in \cite[Section 3]{ApGi} (see also \cite[Section 9]{LahdiliWeighted}) give us:

\begin{prop}\label{vscal-toric}
Suppose that $M$ is toric in the sense of Definition \ref{noncompacttoric} and $\omega$ is an AK metric with symplectic potential $u_\omega$. Set $H = \textnormal{Hess}^{-1}(u_\omega)$ as above. Then the $v$-weighted scalar curvature of $\omega$ is given by 
\begin{equation}
	\Scalv(\omega) = -\sum \left( vH_{ij} \right)_{ij}.
\end{equation}
\end{prop}

In particular, toric solutions to the $(v,\,w)$-cscK equation \eqref{vwcscK} give rise to solutions $u \in C^\infty(P)$ to the \emph{generalized Abreu equation}: 

\begin{equation}\label{genAbreu}
    \sum \left( vH_{ij} \right)_{ij} = - w,
\end{equation}
where as above $\H = \textnormal{Hess}^{-1}(u)$ is the data \eqref{Hdef} associated to $\omega$. Moreover, if $\omega \in \mathcal{H}_{\alpha, T}$ then as we saw above we know that $P \subset \t^*$ is a polyhedron determined uniquely up to translation by the cohomology class $\alpha$.

\subsection{$v$-solitons and Real Monge-Amp\`ere equations}

A special case of particular interest is the notion of a \emph{$v$-soliton} metric, generalizing the notion of a K\"ahler-Ricci soliton. In the compact setting these appear to have been first introduced by Berman-Berndtsson \cite{BB} and Berman-Witt-Nystr\"om \cite{BWN-optimal}, and were recently treated in great detail by Han-Li \cite{HanLi}. By definition, a $v$-soliton is a K\"ahler metric $\omega \in 2\pi c_1(M)$ which satisfies the equation
\begin{equation}\label{vsolitoneq}
	\Ric_\omega - \,\omega = i \p \bp \log v(\mu_\omega).
\end{equation}
In particular, we recover the K\"ahler-Ricci solitons by setting $v(x) = e^{-\ell(x)}$ for $\ell$ a linear function on $\t^*$. These are related to $(v,\,w)$-cscK metrics by:
\begin{lemma}[{\cite[Lemma 2.2]{ApJuLa}}]\label{vsolitonw}
An AK metric $\omega$ is a $v$-soliton if and only if it is a $(v, w)$-cscK metric with  $w = 2\left( n + \langle d \log(v(\mu)) , \, \mu \rangle  \right)v(\mu)$. 
\end{lemma}

Suppose that $\omega$ is a solution to \eqref{vsolitoneq}, and suppose that $\omega = 2i\p\bp \phi$ on the dense orbit. Then there exists an affine linear function $a(\xi)$ such that 
\begin{equation*}
	\log v(\nabla \phi) + 2\phi + \log\det\phi_{ij} = a , 
\end{equation*}
Hence there is a unique affine linear function $a'$ such that after modifying $\phi \mapsto \phi + a'$, $\phi$ satisfies
\begin{equation}\label{rMA1}
	v(\nabla \phi) \det\phi_{ij} = e^{-2\phi}.
\end{equation}
Let $u$ be the corresponding symplectic potential, i.e. $u = L(\phi)$ where $L$ denotes the Legendre transform.  Define a function $\rho_u \in C^\infty(P)$ by 
\begin{equation}\label{rhodef}
	\rho_u(x) = 2\left(\langle \nabla u ,\, x \rangle - u  \right)- \log\det u_{ij}.
\end{equation}
Clearly from \eqref{rMA1} we see that $u$ satisfies 
\begin{equation}\label{rMA2}
	e^{-\rho_u} = v(x).
\end{equation}
By Lemma \ref{propernesslemma}, it follows that if $\omega$ is an AK metric satisfying equation \eqref{vsolitoneq}, we must have that 
\begin{equation}\label{vsoliton-finitevolume}
	\int_P v dx < \infty. 
\end{equation}
For reasons that will become apparent, we define 
\begin{definition}\label{anticanonicallypolarizeddef}
    A toric manifold $M$ is \emph{anticanonically polarized} if $M = M_P$ for Delzant polyhedron of the form 
    \begin{equation}\label{anticanonicalpolyhedron}
        P = \left.\left\{ x \in \t^* \right| \langle x , \, \nu_i \rangle \geq -1, \, \nu_i \in \Gamma, \, i = 1, \dots, N \right\}.
    \end{equation}
\end{definition}
It turns out that this condition is equivalent to the condition that $-K_M$ is ample \cite{CLS}. Moreover, the set $\{\nu_i\}$ are determined uniquely by the anticanonical divisor $D_{-K_M}$ on $M$ once we fix a $\Z$-basis for $\Gamma$. As such, given such an $M$ we will sometimes denote the polyhedron \eqref{anticanonicalpolyhedron} by $P_{-K_M}$. With this in mind, the following is an immediate consequence of the same arguments as in \cite{Don1, uniqueness}.
\begin{lemma}
Let $\omega$ be an AK metric on $M$ which is also a $v$-soliton. If $\mu$ is the moment map normalized by \eqref{normalizedgradientmomentmap}, where $\phi$ is a solution to \eqref{rMA1}, then the image of $\mu$ is precisely $P_{-K_{M}}$. In particular, $M$ is anticanonically polarized. 
\end{lemma}

We also include the following elementary observation, whose proof we did not find in the literature: 
\begin{lemma}\label{vsolitonfutakiinvariantlemma}
Let $P$ be a Delzant polyhedron with a positive weight $v$ which is sufficiently integrable on $P$, and that $w$ is given by \eqref{vsolitonw}. Then 
\begin{equation}\label{vsolitonDFinvariantequiv}
	\int_P \ell w dx = 2\int_P\ell \left( nv + \langle \nabla v , \, x \rangle \right) dx = 2\int_{\p P} \ell v d\sigma -  2\int_P \ell  v dx, 
\end{equation}
for all linear functions $\ell$ on $P$ if and only if $P = P_{-K_M}$ for some anticanonically polarized toric manifold $M$. 
\end{lemma}
What we mean by ``sufficiently integrable'' will be clear in the course of the proof below. A more precise formulation of this type of integrability condition will be given in Section \ref{stability-section}.

\begin{proof}
We begin by observing that if we define the $(n-1)$-form 
\begin{equation}
	\eta = \left(\sum_{j=1}^n (-1)^{j+1} x^j \, dx^1 \wedge \dots \wedge \widehat{dx}^j \wedge \dots \wedge dx^n \right), 
\end{equation}
then 
\begin{equation*}
	d \left(\ell v  \eta  \right) = \ell \left( nv + v_i x^i \right) dx + \ell v dx. 
\end{equation*}
Thus
\begin{equation*}
\begin{split}
	\int_P  \ell \left( nv + v_i x^i \right) dx &= \int_P \ell v \, d\eta  = \int_{\p P} \ell v\, \eta  -  \int_P \ell  v \, dx. 
\end{split}
\end{equation*}
We claim that $\eta$ satisfies 
\begin{equation*}
\left. \eta \right|_{\partial P} = d\sigma
\end{equation*}
if and only if each facet $F$ of $P$ is defined by a linear equality of the form $\langle b , \, x \rangle = -1$, where $b$ is an inner normal to $F$.  To this end, let $F$ be some such facet, and suppose that it has defining equation $\langle b , \, x \rangle = c$. Setting $b = (b_1, \dots, b_n)$, we can assume without loss of generality that $b_1 \neq 0$ so that, for all $x \in F$, 
\begin{equation*}
	x^1 = \frac{c}{b_1} - \sum_{j = 2}^n \frac{b_i}{b_1}x^i. 
\end{equation*}
In particular, for any $j \neq 1$ we have 
\begin{equation*}
\begin{split}
 (-1)^{j+1} dx^1 \wedge \dots \wedge \widehat{dx}^j \wedge \dots \wedge dx^n&=  (-1)^{j+2} \frac{b_j}{b_1}dx^j \wedge dx^2 \wedge \dots \wedge \widehat{dx}^j \wedge \dots \wedge dx^n  \\
 		& = + \frac{b_j}{b_1} dx^2 \wedge \dots \wedge dx^n, 
\end{split}
\end{equation*}
so that 
\begin{equation*}
\begin{split}
	\left. \eta \right|_{F} &= \left(\frac{c}{b_1} -  \sum_{j = 2}^n \frac{b_j}{b_1}x^i \right) dx^2 \wedge \dots \wedge dx^n + \sum_{j=2}^n \frac{b_i}{b_1}x^i dx^2 \wedge \dots \wedge dx^n = \frac{c}{b_1} dx^2 \wedge \dots \wedge dx^n. 
\end{split}
\end{equation*}
Extending $\left.\eta \right|_{F}$ as an $(n-1)$-form on $\R^n$ near $F$, clearly we have 
\begin{equation*}
	d\ell \wedge \left. \eta \right|_{F} = c dx,
\end{equation*}
and so we see that $d\ell \wedge \left. \eta \right|_{F}  = - dx$ if and only if $c = - 1$. 
\end{proof}
As we will see in Section \ref{stability-section}, this amounts to the fact that the weighted Futaki invariant $\mathcal{F}_{v,w}$ associated to the $(v,\,w)$-cscK problem, when restricted to just the linear functions, coincides up to a multiple with the invariant $\mathcal{F}_v:\t \to \R$ given by
\[ \mathcal{F}_v(b) = \int_P \langle x, \, b \rangle v dx. \]
 When $P$ is compact, this was introduced for more general weights by Berman-Berndtsson \cite{BB} and Berman-Witt-Nystr\"om \cite{BWN-optimal}, generalizing the modified Futaki invariant of Tian-Zhu \cite{TZ2}. Indeed, even for unbounded $P$, when $v = e^{-\langle x,\, b \rangle}$ for $b \in C^*(P)$, $\mathcal{F}_v$ coincides with the invariant in \cite{TZ2}, itself an extension of the original invariant introduced by Futaki \cite{Futakiobstruction}. By a simple argument which to the author's knowledge is originally due to Donaldson \cite{Don1}, the existence of a $v$-soliton metric on $M_P$ implies the vanishing of the Futaki invariant.
\begin{lemma}\label{futakivanishes?}
Suppose that $v\in C^\infty(\overline{P}, \R_{>0})$ admits a K\"ahler metric $\omega \in \mathcal{K}_\alpha$ which is a $v$-soliton. Then 
\begin{equation}
	\int_P \ell \, v dx = 0
\end{equation}
for every linear function $\ell$ on $\t^*$.  
\end{lemma}

\begin{proof}
	 Using \eqref{rMA1}, \eqref{rMA2}, we see that under the diffeomorphism $\nabla \phi: \t \to P$,  the volume form $v dx$ satisfies 
	\begin{equation}
		v \, dx = e^{-2\phi} d\xi. 
	\end{equation}
Thus for any $k =0, \dots, n$ we have that
\begin{equation*}
	 \int_P x_k \,  v dx = \int_{t} \phi_k e^{-2\phi} d\xi = - \frac{1}{2}\int_{\t} \frac{\p}{\p \xi_k} \left(e^{-2\phi} \right) d\xi = 0.
\end{equation*}
Here we have used that the boundary term $\int_{\p B_R(0) }e^{-2\phi} d\Theta \xrightarrow{R \to \infty} 0$ since $\phi \geq C^{-1}|\xi| - C$. 
\end{proof}
As we will see in Section \ref{stability-section}, this is also true for the more general $(v,\, w)$-cscK problem under certain assumptions on the weights and the metric. Notably, in the absence of the real Monge-Amp\`ere equation \eqref{rMA1}, there is no analog of the a priori integrability condition \eqref{vsoliton-finitevolume}.

\subsection{K\"ahler Cones}

We summarize some of the fundamentals of Sasakian geometry, with an emphasis on the context considered in this paper. There are many equivalent formulations, we primarily stick to the description of \cite{MSY, MSY2} (see also \cite{SunHe-Frankel, FOW}), as it is particularly suitable for our present setting. We do not attempt to give a complete description, only briefly mention those results that we will use later on. For all details we refer to the sources mentioned above.

A \emph{Riemannian cone} $(M_c,  \, g_c)$ is a smooth manifold $M_c \cong \R_+ \times L$ for a compact smooth manifold $L$ with a metric of the form 
\[ g_c = dr^2 + r^2 g_L,\]
 where $g_L$ is the pullback of a Riemannian metric on the $L$. Given $(M_c, \, g_c)$ we often identify $L$ with the set $\{r = 1\} \subset M_c$. We say that the metric $g_c$ is a \emph{K\"ahler cone} metric if $M_c$ is given a complex structure $J_c$ with respect to which $g_c$ is K\"ahler. In this case we also say that that the metric $g_L$ on $L$ is \emph{Sasakian}. Then there is an associated \emph{Reeb field} $K_c = J_c(\rad)$ which is tangent to $L$. Both $\rad$ and $K_c$ are real holomorphic and $K_c$ is $g$-Killing \cite[Appendix A]{MSY2}. In fact, it was proved by van Coevering that any K\"ahler cone is biholomorphic to the regular part of an affine variety \cite[Theorem 3.1]{VC2}.


A K\"ahler cone $(M_c,\, J_c, \, g_c)$ is \emph{toric} if $(M_c, \, J_c)$ is toric in the sense of Definition \ref{noncompacttoric} and if the action of the real torus $\T$ restricts to an effective action on $L$ and if the Reeb field satisfies $K_c \in \t$. We then have a natural identification between $\t$ and $\t_L$, representing the vector spaces of fundamental vector fields for the $\T$-action on $M_c$ and $L$, respectively. Toric K\"ahler cones have been studied extensively, particularly in the search for complete Ricci-flat metrics and related ideas \cite{MSY, MSY2, FOW, Ler, VC1,VC2}. 

K\"ahler cone metrics are always exact, in the sense that 
\begin{equation}\label{cone-exact}
    \omega_c = \frac{i}{2}\p\bp r^2,
\end{equation}
globally on $M_c$. Therefore, restricting to the dense orbit $\Cstarn \subset M_c$, we have a natural normalization for the choice of K\"ahler potential $\phi$ from Proposition \ref{guillemin-exact}, namely $\phi = r^2/4$. We continue to use the convention that the moment map $\mu_c$ is normalized to be equal to $\nabla \phi$ when restricted to the dense orbit. With this choice, it follows that the image of $\mu_c$ is a rational polyhedral cone 
\begin{equation}
    \mu_c(M_c) = C := \left\{ x \in \t^* \: | \: L_{\nu_i}(x) \geq 0, \, i = 1, \dots, N_C \right\},
\end{equation}
and that the hamiltonian potential for $K_c$ is given by 
\begin{equation*}
    \langle \mu_c, \, K_c \rangle = \frac{r^2}{2} = 2\phi.
\end{equation*}
Just as in the previous section we use the symplectic coordinate system on $C$ determined by the moment map. Here we have that the radial vector field satisfies 
\begin{equation}\label{radial-poly}
    \rad = 2 x^i \frac{\p}{\p x^i}, 
\end{equation}
and as before the Guillemin potential is 
\begin{equation}\label{guilleminpot-cone}
    u_C = \frac{1}{2}\sum_{i = 1}^{N_C} L_{\nu_i}(x) \log( L_{\nu_i}(x) ). 
\end{equation}
Set $L_\infty(x) = \sum_{i = 1}^{N_C}L_{\nu_i(x)}$. Then for any $b \in \t$, we define 
\[ u_b = \frac{1}{2}L_b(x) \log(L_b(x)) -  \frac{1}{2} L_\infty(x) \log(L_\infty(x)). \]
We have the following from \cite{MSY}:
\begin{lemma}\label{MSY-structure}
    Let $\omega_c$ be a toric K\"ahler cone metric on $M_c$. Then the Reeb vector field $K_c$ satisfies $K_c = Y_b$ for some $b \in C^*$. Moreover, the symplectic potential $u$ associated to $\omega_c$ has the property that 
    \begin{equation}\label{symplecticpotential-cone}
        u = u_C + u_b + h, 
    \end{equation}
    where $h \in C^\infty(\overline{C})$ is homogeneous of degree 1 on $C$. In particular, the data $\G = \textnormal{Hess}(u)$ is homogenous of degree $-1$ and $\H = \textnormal{Hess}^{-1}(u)$ is homogeneous of degree 1. 
\end{lemma}

\section{Weighted stability}\label{stability-section}

\subsection{Preliminaries}

We will use the following repeatedly in this section.
\begin{definition}[{\cite{DonStabTor}}]\label{Pdelta}
	For all $\delta > 0$ sufficiently small, let $P_\delta$ be the interior polyhedron to $P$ with facets parallel to those of $P$ separated by a distance $\delta$. 
\end{definition}

 Given a domain $\Omega \subset \overline{P}$, $\beta \in \R$, and $v \in C^\infty(P)$, we define a weighted function space $C^k_{\beta}(\Omega) \subset C^k(\Omega)$ to be those functions $f$ with $k$ continuous derivatives on the interior of $P$ such that 
\begin{equation}\label{Ckbeta}
	|| f ||_{C_\beta^k(\Omega)} = \sum_{|\alpha| \leq k} \sup_\Omega \left|v^{\beta} \p_{x_{\alpha}}f \right|.
\end{equation}

 \begin{definition}\label{exponential-decay}
 	We say that the weights $(v,\, w)$ are \emph{exponentially decaying with derivatives} on $P$ if the following are satisfied. First, we ask that for every $k =0 , 1 , \dots $, there exist constants $C_1(k), C_2 > 0$ such that 
\begin{equation}\label{exponentialdecay-eqn1}
\begin{split}
  |v|  &\leq C_1(0) e^{-C_2|x|},  \textnormal{ and } 
  \sum_{|\alpha| \leq k} \left| v^{-1} \p_{x_{\alpha}}v \right|  \leq C_1(k) \textnormal{ for }  k \geq 1.
\end{split}
\end{equation}
Note that the second condition is precisely the condition that $||v||_{C^k_{-1}(P)} < C(k)$. Then we ask that there exists $\beta^* > 0$ and constants $C_3(k), k = 0, 1, \dots $ such that 
\begin{equation}
    ||w||_{C^{k}_{-\beta^*}(P)} < C_3(k).
\end{equation}
In this case we say that $(v, \, w) \in \mathcal{W} = \mathcal{W}(P)$.
 \end{definition}
 As we will see in Section \ref{section-KRS}, this is a natural class of examples from the perspective of shrinking gradient K\"ahler-Ricci solitons. In this case, the weights $v, w$ are both of the form $p(x)e^{-\langle x, \, b \rangle}$ for $b \in C(P)$ and $p(x)$ a positive rational function, and hence satisfy \eqref{exponentialdecay-eqn1} for any $\beta^* < 1$. In some practical applications, it is impossible to take $\beta^* = 1$, see for example \ref{toricSGKRSweight},\ref{fibrationsolitons}.
 \begin{remark}
     Many of the results of this paper do not require the full exponential decay \eqref{exponentialdecay-eqn1} of the weights, and could be modified to work for more general weights $(v, \, w)$. In particular, all of the results of the current section and Section \ref{section-functionspaces} can be generalized to weights with sufficiently fast polynomial decay. Nonetheless we will stick to the exponential case for clarity of exposition, as the introduction of more general rates requires a more careful analysis of the various parameters that will arise in future sections. 
 \end{remark}

Here we introduce the technical condition on the metric under which we will work.  
\begin{definition}\label{asymptotic-metric-general}
	Let $\omega \in \alpha$ be an AK metric on $M$ with polyhedron $P \subset \t^*$ containing the origin in its interior, corresponding data $\H$ \eqref{Hdef} and symplectic potential $u \in C^\infty(P)$.  Let $v \in \mathcal{W}$ be a weight satisfying the conditions of Definition \ref{exponential-decay}. We say that $\omega$ is in the class $\scaryH$ if there exists a single $\varepsilon \in [0,  \min\{ \beta^*, \frac{1}{2}\})$ such that all of the following hold:
	\begin{enumerate}
		\item\label{asymptotics1}  $\sup_{p \in M}v(\mu_{\omega})^{\varepsilon}|| \mathbf{H} ||_{\t}^2 < \infty,$ 
		\item\label{asymptotics2} $\sup_{p \in M}v(\mu_{\omega})^{\varepsilon}||d \mathbf{H} ||_{\t}^2  < \infty,$    \item\label{asymptotics4} there exists a $\bar{\delta} > 0$ and a $C > 0$ such that  $\sup_{x \in \overline{P}\backslash P_{\bar{\delta}}} \left| \frac{\p^2 H_{ij}}{\p x^k \p x^l } \right|^2 < C$ for all $i,j,k,l = 1, \dots, n$,
		\item\label{asymptotics3}  $||u||_{C^0_{\varepsilon}(P)} < \infty$ and  $||u||_{C^2_{\varepsilon}(P_\delta)} < \infty$ for all $\delta > 0$ sufficiently small.
	\end{enumerate} 
 Here $||d \mathbf{H} ||_{\t}^2 := \sum_{i,j,k}\left| \frac{\p}{\p x^k}H_{ij} \right|^2 = \sum_{i,j,k}\left| H_{ij,k} \right|^2$.
\end{definition} 		
In general, as a slight abuse of notation, we will say that $\H \in \scaryH$ for a positive-definite $\H = (H_{ij})$ on $P$ if it satisfies the conditions of Definition \ref{asymptotic-metric-general} with respect to $g = G_{ij}dx^i \otimes dx^j$ on $P$. Moreover we will say that a convex function $u\in C^\infty(P)$ lies in $\scaryH$ if $\H \in \scaryH$ for $\H = \textnormal{Hess}^{-1}(u)$. The following is clear from item \ref{asymptotics2}:

\begin{lemma}\label{hardterm}
	Suppose that $\H$ satisfies the conditions of Definition \ref{asymptotic-metric-general} and $v$ is an exponentially decaying weight as in Definition \ref{exponential-decay}.  Then there exists a $C > 0$ such that for any $i,j = 1, \dots, n$ we have 
	\begin{equation*}
		\sup_{P} \left( v^{\varepsilon - 1} \sum_{k} | (vH_{ij})_k | \right) < C
	\end{equation*}
  In particular, for every $j = 1, \dots, n$ 
  \begin{equation*}
		\sup_{P} \left( v^{\varepsilon - 1 } \sum_{i} | (vH_{ij})_i | \right) < C.
	\end{equation*}
\end{lemma}

\subsection{Function spaces and K-stability}\label{section-functionspaces}
In this section we will fix a Delzant polyhedron $P \subset \t^*$ and an exponentially decaying weight $v \in \mathcal{W}(P)$ as in Definition \ref{exponential-decay}. Let $\beta \in (\varepsilon, 1 - \varepsilon)$, which by our assumption that $0 \leq \varepsilon  < \frac{1}{2}$ is nonempty. 

We define the set $\Cee_\beta$ to be the set of convex functions $f \in C^\infty(P) \cap C^0(\overline{P})$ such that 
\[ \begin{array}{ccl}

           f \in \Cee_\beta &\Rightarrow &  || f ||_{C_\beta^0(\overline{P})} < \infty \textnormal{ and }  || f ||_{C_\beta^1(\overline{P}_{\delta})} < \infty  \textnormal{ for all } \delta << 1.
\end{array}\] 
We set $\Cee_\beta^* \subset \Cee_\beta$ to be the set of \emph{normalized} functions, i.e. those which satisfy 
\begin{equation}
	\Cee_\beta^* = \{ f \in \Cee_\beta \: | \: f \geq f(0) = 0 \},
\end{equation}
recalling that we always normalize $P$ such that $0$ lies in the interior. 

\begin{lemma}\label{spacescompare-1}
    Suppose $u \in C^\infty(P) \cap C^0(\overline{P})$ is the symplectic potential corresponding to an AK metric $\omega \in \scaryH$, and $v \in \mathcal{W}$ be an exponentially decaying weight. Then $u \in \Cee_\beta$ for any $\beta \geq \varepsilon$. 
\end{lemma}

\begin{proof}
This is clear from condition \ref{asymptotics3} of Definition \ref{asymptotic-metric-general}.
\end{proof}

The main goal of this section is to prove the following analog of \cite[Lemma 3.3.5]{DonStabTor}:
\begin{lemma}\label{ibplemma}
Suppose that the weights $(v,\,w)\in \mathcal{W}$ and that $\H$ is the data \eqref{Hdef} corresponding to a $(v, w)$-weighted cscK metric $\omega$ on $M$ satisfying the conditions of Definition \ref{asymptotic-metric-general}. Then, for all $f \in \Cee_\beta$, we have 
\begin{equation}\label{ibpformula}
\begin{split}
	\int_{P} (vH_{ij}) f_{ij} dx  &=  \int_P(vH_{ij})_{ij}f dx + 2\int_{\p P} f v d\sigma \\
						& =  2\int_{\p P} f v d\sigma -  \int_{P} f w dx. 
\end{split} 
\end{equation}
\end{lemma}

\begin{remark}
    In proving Lemma \ref{ibplemma}, we primarily use the $(v,\,w)$-cscK assumption on our metric $\omega$ to ensure that $\Scalv(\omega)$ is sufficiently integrable on $P$. By the same proofs that follow, \eqref{ibpformula} holds without the condition that $\omega$ is $(v,\,w)$-cscK as long as $\Scalv(\omega)$ decays sufficiently rapidly. 
\end{remark}

Let $dx^{\widehat{j}} = (-1)^{j+1}dx^1 \wedge \dots \wedge d\widehat{x}^{j}\wedge \dots \wedge dx^n $, so that in particular $dx^j \wedge dx^{\widehat{j}} = +dx$.

\begin{lemma}\label{regularboundaryterm1}
	For all $f \in \Cee_\beta$, we have 
	\begin{equation*}
		\int_{\p P_\delta}   (vH_{ij}) f_{i} - (vH_{ij})_{i}f  dx^{\widehat{j}} < \infty. 
	\end{equation*}
\end{lemma}
\begin{proof}
  We have 
 \begin{equation*}
\begin{split}
		\left| \int_{\p P_\delta}   (vH_{ij}) f_{i} - (vH_{ij})_{i}f  dx^{\widehat{j}} \right| &\leq \int_{\p P_\delta}  \left|(vH_{ij})_{i}f\right| + \left|(vH_{ij}) f_{i}\right| dx^{\widehat{j}} \\
			& \leq  2 C(\delta) \int_{\p P_\delta}   v^{1-\varepsilon-\beta}  dx^{\widehat{j}} < \infty
\end{split}
	\end{equation*}
by Definition \ref{exponential-decay} and Lemma \ref{hardterm}, since $\varepsilon + \beta < 1$.
\end{proof}

Let $b_{+} \in C^*$ be arbitrary,  and for $\delta^* > 0$, let $H_{\delta^*}$ be the half space 
\begin{equation*}
	H_{\delta^*} = \{x \in \t^* \: | \: \langle b_{+}, \, x \rangle \leq \delta^*  \}.
\end{equation*}
	Since $b_{+} \in C^*$, it follows that for any $\delta$ sufficiently small that the interior polyhedron $P_{\delta}$ is nonempty, the polytope $Q_{\delta, \delta^*} = P_{\delta} \cap H_{\delta^*}$ is bounded.  As a shorthand, set $\p H_{\delta,  \delta^*} := \overline{P}_{\delta} \cap \p H_{\delta^*}$. Then we have
\begin{lemma}\label{extraboundaryterm}
	For every $\delta > 0$,  and for every $f \in \Cee_\beta$, we have
\begin{equation*}
	\lim_{\delta^* \to \infty}	\int_{\p H_{\delta, \delta^*}}      (vH_{ij}) f_{i} - (vH_{ij})_{i}f dx^{\widehat{j}}  = 0. 
\end{equation*}
\end{lemma}
\begin{proof}
The proof is very similar to that of Lemma \ref{regularboundaryterm1}.  Indeed, on all of $P_\delta$, we have 
	\begin{equation*}
		\left|    (vH_{ij}) f_{i} - (vH_{ij})_{i}f   \right| \leq 2C(\delta)  v^{1-\varepsilon-\beta}.
	\end{equation*}
	  The result now follows since $v$ is exponentially decaying, $\varepsilon + \beta < 1$, and for fixed $\delta$ the volume of $\p H_{\delta, \delta^*}$ is a polynomial in $\delta^*$ of degree $n_C-1$, where $n_C = \dim \t_C$ (see Lemma \ref{makeavariety}). 
\end{proof}

\begin{corollary}\label{ibpPdelta}
	For every $\delta > 0$, 
\begin{equation}\label{ibpformuladelta}
	\int_{P_\delta} (vH_{ij}) f_{ij} - (vH_{ij})_{ij}f dx  = 	\int_{\p P_\delta}    (vH_{ij}) f_{i} - (vH_{ij})_{i}f  dx^{\widehat{j}}, 
\end{equation}
 for every $f \in \Cee_\beta$.
\end{corollary}
\begin{proof}
	For every $\delta, \delta^* > 0$,  we know that 
\begin{equation*}
\begin{split}
	\int_{Q_{\delta, \delta^*}} (vH_{ij}) f_{ij} - (vH_{ij})_{ij}f dx  &= 	\int_{\p Q_{\delta, \delta^*}}    (vH_{ij}) f_{i} - (vH_{ij})_{i}f   dx^{\widehat{j}}.  \\
				& =  	\int_{\p P_\delta \cap H_{\delta^*}}     (vH_{ij}) f_{i} - (vH_{ij})_{i}f  dx^{\widehat{j}} +  	\int_{\p H_{\delta,\delta^*}}     (vH_{ij}) f_{i} - (vH_{ij})_{i}f   dx^{\widehat{j}}.  
\end{split}
\end{equation*}
Taking $\delta^* \to \infty$ gives \eqref{ibpformuladelta}. 
	
\end{proof}

\begin{proof}[Proof of Lemma \ref{ibplemma}]
	The proof proceeds as in \cite[Lemma 3.3.5]{DonStabTor}. The key fact that we must establish is that 
 \begin{equation}\label{ibplemmaproofmainlimit}
     \lim_{\delta \to 0} \int_{\p P_\delta} (vH_{ij}) f_{i}dx^{\widehat{j}} = 0.
 \end{equation}
 To see this, we use the same strategy as in \cite{DonStabTor}. Fix a facet $F_\delta$ of $P_\delta$ with unit inner normal $\nu$, and write 
 \[ \int_{\p F_\delta} (vH_{ij}) f_{i}dx^{\widehat{j}} = -\int_{F_\delta} v \p_{\eta}f d\sigma, \]
 with $\eta$ given by $\eta_i = \sum_j H_{ij}\nu_j$. Fix a point $x \in F_\delta$, let $F$ be the facet of $P$ with inner normal $\nu$, and set $\hat{y} \in F$ to be the point $\hat{y} = x - \delta \nu$. We consider the Taylor approximation to $\eta(x)$ based at $\hat{y}$, so that 
  \begin{equation*}
    \begin{split}
      \eta(x) &= \eta(\hat{y}) + d\H_{\hat{y}}(\nu, x - \hat{y}) + R(x). 
    \end{split}
  \end{equation*}
  By Proposition \ref{boundaryconditions}, $\eta(\hat{y}) = \H_{\hat{y}}(\nu,\, \cdot) = 0$, whereas $d\H_{\hat{y}}(\nu, x - \hat{y}) = \delta d\H_{\hat{y}}(\nu, \nu) = 2\delta \nu$. Therefore 
  \begin{equation}\label{eta-expansion}
      \eta(x) = 2\delta \nu + R(x). 
  \end{equation}
  By condition \ref{asymptotics4} of Definition \ref{asymptotic-metric-general}, as long as $\delta \leq \bar{\delta}$, then there exists a constant $C$ such that 
  \[ \left|R(x) \right| < C\delta^2. \]
  Next we let $y = y(x)$ be the point in $\p P$ closest to $x$ which lies on the line $x - \lambda \eta$, $\lambda \in \R$. By \eqref{eta-expansion}, we see that as long as $\delta$ is sufficiently small, then we will have that $y(x) \in F$ for any $x \in F_{\delta}$. In particular 
  \begin{equation*} 
  \begin{split}
      -a = \langle y, \nu \rangle &= \langle x,\, \nu \rangle - \lambda \langle \eta, \, \nu \rangle \\
        & = -a + \delta - 2\lambda |\nu|^2 - \lambda \langle R(x), \, \nu \rangle,
  \end{split} 
  \end{equation*}
  so that 
  \begin{equation*}
      \begin{split}
          | y - x| = \left| \lambda \right| = \left| \frac{\delta}{2 |\nu|^2 +  \langle R(x), \, \nu \rangle } \right| \leq \frac{\delta}{2 |\nu|^2 -  C\delta^2 |\nu| } \leq C \delta,
      \end{split}
  \end{equation*}
  as long as $\delta$ is sufficiently small. Then by condition \ref{asymptotics2} of Definition \ref{asymptotic-metric-general}, there exists a constant $C$ independent of $x$ such that 
 \[v \left|\eta(y) - \eta(x) \right| \leq C v^{1-\varepsilon}|y-x|. \]
 By the same reasoning as above, it follows from the local boundary behavior of $\H$ given by Proposition \ref{boundaryconditions} that $\eta(y) = 0$. Therefore $v|\eta(x)| < Cv^{1-\varepsilon}|y-x|$, and so by the convexity of $f$ we have 
 \[ v | \p_\eta f| \leq C v^{1-\varepsilon}|f(y) - f(x)|, \]
 again for a constant $C$ independent of $x$.

 Then we have 
 \begin{equation*}
 \begin{split}
     \left| \int_{F_\delta} (vH_{ij}) f_{i}dx^{\widehat{j}}\right| &= \left| \int_{F_\delta} v \p_{\eta}f d\sigma, \right| \leq C \int_{F_\delta} v^{1-\varepsilon}|f(y) - f(x)| d\sigma  \\
        &  \leq C\max_{x \in F_{\delta}}\{|v^{\beta'}f(y) - v^{\beta'}f(x)| \} \int_{F_\delta} v^{1-\varepsilon - \beta'} d\sigma \\
        &\leq C \max_{x \in F_{\delta}}\{|y - x|  \} \\
        &\leq C \delta,
\end{split}
 \end{equation*}
for some choice of $\beta < \beta' < 1$, owing to the fact that $v^{\beta'}f$ is uniformly continuous on $\overline{P}$. This establishes \eqref{ibplemmaproofmainlimit}. The convergence of the remaining terms follows as in \cite{DonStabTor}, using condition \ref{asymptotics1} of Definition \ref{asymptotic-metric-general}, together with Proposition \ref{boundaryconditions} and the uniform continuity of $v^{\beta'}f$ for $\beta < \beta' < 1$.
 \end{proof}



Given \eqref{ibpformula}, then following \cite{DonStabTor, LahdiliWeighted, LiLianSheng1, simonYTD} we define:
\begin{definition}\label{DonaldsonFutakidef}
Let $P$ be a Delzant polyhedron and $(v,\,w)$ exponentially decaying weights. We define the \emph{weighted Futaki invariant} $\mathcal{F}_{v,w}: \Cee_\beta \to \R$ by
    \begin{equation}\label{DonaldsonFutakieqn}
        \mathcal{F}_{v,w}(f) =  2\int_{\p P} f v d\sigma - \int_{P} f w dx.
    \end{equation}
\end{definition}

\begin{definition}\label{PL-and-stable-def}
    We say that a function $f \in C^0(\overline{P})$ is \emph{convex piecewise-linear} if 
    \begin{equation*}
        f = \max\{ \ell_1, \dots, \ell_k\},
    \end{equation*}
    for finitely many affine-linear functions $\ell_1, \dots, \ell_k$. We say that $P$ is \emph{$(v,\,w)$ K-semistable} if 
    \begin{equation}\label{semistable-nonneg}
        \mathcal{F}_{v,w}(f) \geq 0
    \end{equation}
    for all convex piecewise-linear functions on $P$. Moreover we say that $P$ is \emph{$(v,\,w)$ K-stable} if it is K-semistable and equality holds in \eqref{semistable-nonneg} if and only if $f = \ell_1$ is affine-linear.  
\end{definition} 

With this in place, we can now prove Theorem \ref{Mtheorem-existenceimpliesstable}. The proof, originally due to \cite{ZhouZhu} is the same as in the compact case (see also \cite{vestisnotes}) once we have Lemma \ref{ibplemma}, but we include it here for the convenience of the reader. 
\begin{prop}\label{existenceimpliesstable}
    Let $P$ be a Delzant polyhedron and suppose that $(v,\, w)$ are weights in the class $\mathcal{W}$. Suppose that there is a solution $u \in \scaryH$ to the weighted Abreu equation 
    \[ \sum_{i,j}\left( v H_{ij} \right)_{ij} = - w,\]
    where $\H = (H_{ij}) = \textnormal{Hess}^{-1}(u)$. Then $P$ is $(v,\,w)$ K-stable.     
\end{prop}

\begin{proof}
Suppose that $f$ is convex piecewise-linear, so that we can write 
    \[ P = \bigcup \Delta_a, \]
    where $\ell_a := f|_{\Delta_a}$ is affine-linear. We write 
    \[ \p\Delta_a =  \p P_a \cup \bigcup F_{ab} ,  \] 
    where $\p P_a$ is the corresponding portion of $\p P$ and $F_{ab}$ meet the interior of $P$. Note that $F_{ab}$ are defined by the equation $L_{ab}(x) := \ell_a(x) - \ell_b(x) = 0$. Then by Corollary \ref{ibpPdelta} applied to each $\Delta_a$ and Lemma \ref{ibplemma}, we have that 
    \begin{equation*}
    \begin{split}
        - \int_{\Delta_a} f w dx &= \int_{\Delta_a} \ell_a \sum_{i,j}\left( v H_{ij} \right)_{ij}dx \\
            &= -2\int_{\p P_a} f v d\sigma + \sum_{b} \int_{F_{ab}} (vH_{ij}) \ell_{a,i} - (vH_{ij})_{i}\ell_a dx^{\widehat{j}} \\
            & = -2\int_{\p P_a} f v d\sigma + \sum_{b} \int_{F_{ab}} \left( (vH_{ij}) \ell_{a,i} - (vH_{ij})_{i}\ell_a \right) L_{ab,j} d\sigma ,
    \end{split}
    \end{equation*}
    since $f_{ij} \equiv 0$. Here we have extended the definition of $d\sigma$ in the obvious manner to define a positive measure on $F_{ab}$. In this way, the last line above is justified using the convexity of $f$ to conclude that $d\ell_a$ defines an inward normal to $F_{ab}$ on $\Delta_a$ whereas $d\ell_b$ defines an outward normal. By construction we have that $L_{ba} = - L_{ab}$ and $\ell_a|_{F_{ab}} = \ell_b|_{F_{ab}}$, from which we can see that 
    \begin{equation}\label{piecewiseibp}
    \begin{split}
        -\int_{P} f w dx &= -\sum_{a} \int_{\Delta_a} \ell_a w dx \\
                & = -\sum_{a}2\int_{\p P_a} \ell_a v d\sigma +\sum_{a,b} \int_{F_{ab}} (vH_{ij}) \ell_{a,i}  L_{ab,j} d\sigma \\
                & = -2\int_{\p P} f v d\sigma + \sum_{a,b} \int_{F_{ab}} (vH_{ij}) \ell_{a,i}(\ell_{a,j} - \ell_{b,j}) d\sigma  \\
                & = -2\int_{\p P} f v d\sigma + \sum_{a < b} \int_{F_{ab}} (vH_{ij}) (\ell_{a,i} - \ell_{b,i})(\ell_{a,j} - \ell_{b,j}) d\sigma.
    \end{split}
    \end{equation}
Since $\H$ is positive-definite, we then have that $\mathcal{F}_{v,w}(f) \geq 0$. Moreover, if  $\mathcal{F}_{v,w}(f) = 0$, then in fact the set of ``creases'' $\{F_{ab}\}$ must be empty, and hence $f = \ell_1$ is actually affine.
\end{proof}

\begin{remark}\label{remarkdsigma-notquadratic}
Note that by definition of the measure $d\sigma$ on $F_{ab}$, the last term in \eqref{piecewiseibp} scales appropriately as we scale $f$. That is to say, if for $K > 0$ we scale $f \mapsto Kf$, then the corresponding measure scales as $d\sigma \mapsto K^{-1}d\sigma$.
\end{remark}

\subsection{Nonexistence} \label{section-nonexistence}

\begin{prop}\label{specialvwexample}
    Let $P = \R^2_{\geq -1}$ and $v, w$ be weights of the form 
    \begin{equation}\label{specialvwexample2}
    v = q_1(x) e^{-(x_1 + x_2)}, \hspace{.5in} w = q_2(x) e^{-\lambda(x_1 + x_2)}, 
    \end{equation}
    where $q_1, q_2$ are positive symmetric rational functions on $P$, $\lambda \in (0,1)$. Then even if
        \[\mathcal{F}_{v,w}(f) = 2\int_{\p P} f v d\sigma - \int_{P}f w dx = 0 \] 
    for all affine-linear $f$, $\C^2$ does not admit any $(v,\,w)$-cscK metric satisfying \eqref{asymptotic-metric-general}.
\end{prop}

\begin{proof}
    For $R > 0$, set 
    \[ f_R(x) = \left\{ \begin{array}{lr} x_1 + x_2 - R & x_1 + x_2 \geq R \\ 0 & x_1 + x_2 \leq R \end{array}\right. . \]
    For any fixed $\varepsilon > 0$ sufficiently small, we will have that 
    \[ q_1(x) e^{-\varepsilon(x_1 + x_2)} \leq 1, \hspace{.2in} q_2(x) e^{ \varepsilon (x_1 + x_2)} \geq 1 \]
    for $x_1 + x_2 \geq R$, for all $R$ sufficiently large. It is then straightforward to compute that 
    \[ \int_{\p P} f_R v d\sigma \leq \int_{\p P} f_R e^{-(1-\varepsilon)(x_1 + x_2)} d\sigma = O\left(e^{-(1-\varepsilon)R} \right), \]
    whereas 
    \[ \int_{ P} f_R w dx \geq  \int_{ P} f_R e^{-(\lambda+\varepsilon)(x_1 + x_2)} dx = O\left(Re^{-(\lambda+\varepsilon)R} \right). \]
    For any $\lambda \in (0, 1)$, we can choose an $\varepsilon$ such that $\lambda + \varepsilon \leq 1 - \varepsilon$. Thus, for a sufficiently large choice of $R$ we will have that $a$ times the latter dominates twice the former, and hence $\mathcal{F}_{v,aw}(f_R) < 0$.
\end{proof}

\begin{example}
    We give an example of weights $(v,\,w)$ of the form \eqref{specialvwexample2} which also satisfy 
    \begin{equation}\label{simultaneously-vanishing}
        \int_{\p P} \ell v d\sigma = \int_{P} \ell w dx = 0,
    \end{equation}
    for all linear $\ell$ on $P$. Once we have \eqref{simultaneously-vanishing}, then for an appropriate positive multiple $a > 0$ we will have $\mathcal{F}_{v,aw}$ vanishes on the affine-linear functions. Indeed, since we have \eqref{simultaneously-vanishing}, in order to find such an $a > 0$ we only need to choose $a$ such that $\mathcal{F}_{v,aw}(1) = 0$. Thus we set 
    \[ a = 2\frac{\int_{\p P} v d\sigma}{\int_P w dx}.\]
    To find a suitable $q_1,q_2$, we first claim that \eqref{simultaneously-vanishing} holds for $v$ if we set 
    \[q_1(x_1,x_2) = x_1^2x_2^2+1.\]
    Indeed, note that
    \[ \int_{-1}^\infty (x^2 + 1)(x - 1) e^{-x} dx = 0. \]
    Then we compute 
    \begin{equation*}
    \begin{split}
        \int_{\p P} x_1 q(x_1,x_2)e^{-(x_1 + x_2)} dx &= e\int_{-1}^\infty x_1 q(x_1,-1)e^{-x_1} dx_1 - e\int_{-1}^\infty  q(-1,x_2)e^{-x_2} dx_2 \\
            &=e\int_{-1}^\infty (x - 1)(x^2+1) e^{-x}dx = 0,
    \end{split}
    \end{equation*}
    and symmetrically for $x_2$. Hence \eqref{simultaneously-vanishing} holds for $v$ with this choice of $q_1$. To find an appropriate $w$, one can observe that for $c > 0$, the function
    \[ c \mapsto \int_{-1}^\infty \frac{x}{c + x^4}  e^{-\lambda x}dx\]
    changes sign at least once. We denote $c_\lambda \in (0, \infty)$ a point where the integral vanishes. Then it follows that 
    \[ \int_{P} \left(\frac{x_i}{(c_\lambda + x_1^4)(c_\lambda + x_2^4)} \right) e^{-\lambda( x_1 + x_2)} dx_1dx_2 = 0 \]
    for $i = 1,2$, and therefore \eqref{simultaneously-vanishing} holds for $q_2(x) = [(c_\lambda + x_1^4)(c_\lambda + x_2^4)]^{-1} $. Then with these choices, $\mathcal{F}_{v,aw}$ vanishes on the affine-linear functions, but by Proposition \ref{specialvwexample} there is no $(v,\,w)$-cscK metric on $\C^2$ satisfying \eqref{asymptotic-metric-general}. In particular, as we will see in Section \ref{section-productmodels} (c.f. Theorem \ref{asym-product-satisfiesconditions}), if such a metric exists it is not asymptotically conical in the sense of Definition \ref{asymptoticsproduct}. 
\end{example}

\begin{remark}
    We take a moment to comment on the choice of weights in Proposition \ref{specialvwexample}. Clearly it also holds for $q_1,q_2$ not necessarily rational, as long as $q_1$ is positive, $q_2$ is \emph{eventually} positive, and the growth rate is controlled. Note that the Euclidean metric on $\C^2$ is $v$-soliton with $v = e^{-(x_1 + x_2)}$. In this case $w$ is given by Lemma \ref{vsolitonw} as $w = 2(n - (x_1 + x_2))e^{-(x_1 + x_2)}$, so in the context of Proposition \ref{specialvwexample} we have $\lambda = 1$, but even more crucially this is not eventually positive on $P$. Indeed for arbitrary polyhedra $P$, the weight $w$ associated to $v$-soliton by Lemma \ref{vsolitonw} where $v = e^{-\langle x ,\, b \rangle}$ with $b \in C(P)^*$ will be eventually negative on $P$.
\end{remark}

\subsection{Uniform stability}\label{section-uniform}

We will continue to use the notation of Section \ref{section-functionspaces}, with the additional shorthand that $Q_\delta = Q_{\delta, \delta^{-1}}$.

\begin{definition}\label{Kstab-def}
    Let $P$ be a Delzant polyhedron in $\t^*$, and $(v,\,w) \in \mathcal{W}$. We set 
     \begin{equation}\label{ceekbeta}
         \Cee_{\beta}^*(K) = \left\{ f \in \Cee_{\beta}^* \left| \int_{P\backslash H^{\delta^*}} f v^\beta \leq K (\delta^*)^{-1}\right\}\right.
     \end{equation}
     We say that $P$ is \emph{$(v,\,w)$-uniformly K-stable} with parameters $(K,\beta, \gamma)$ if there exists a $\lambda_{K,\beta, \gamma} > 0$ such that 
    \begin{equation}\label{Kstab-ineq-1}
        \mathcal{F}_{v,w}(f) \geq \lambda_{K,\beta,\gamma} \int_{ P} f v^{\gamma} dx
    \end{equation}
    for all $f \in \Cee_\beta^*(K).$ In particular, we say that $P$ is \emph{$\beta$-weighted K-stable} if for all $K > 0$, there is a $\lambda_\beta(K) > 0$ such that
     \begin{equation}\label{Kstab-ineq-2}
        \mathcal{F}_{v,w}(f) \geq \lambda_\beta(K) \int_{P} f v^{\beta} dx
    \end{equation}
    for all $f \in \Cee_\beta^*(K).$
\end{definition}

\begin{remark}\label{remark-normcompare}
    A few remarks about the definition. One way to think about condition \eqref{ceekbeta} is as a weak version of an a priori assumption on the asymptotics of $f$, as follows. Suppose that $\beta_1 \in (\varepsilon, \beta)$. Then as we saw in Lemma \ref{spacescompare-1}, if $u\in C^\infty(P)$ is a symplectic potential for a metric $\omega \in \scaryH$, then $u \in \Cee_{\beta_1}$. After normalizing we may assume that $u \in \Cee_{\beta_1}^*$. Since $\beta_1 < \beta$, the set $\D(K_1)$ of all $f\in \Cee^{*}_\beta$ with $||f-u||_{C^0_{\beta_1}} < K_1$ is a convex subset of $\Cee_\beta^*$. Then for any $f \in \D(K_1)$ we will have 
    \begin{equation*}
    \begin{split}
        \int_{P\backslash H_{\delta^*}} f v^\beta dx &\leq  \int_{P\backslash H_{\delta^*}} u v^\beta dx + \int_{P\backslash H_{\delta^*}} |f - u| v^{\beta} dx  \\ 
        & \leq (C(u) + K_1) \int_{P\backslash H_{\delta^*}} v^{\beta - \beta_1} dx \leq K(u,K_1,\beta_1) (\delta^*)^{-1},
     \end{split}   
    \end{equation*}
    and therefore $\D(K_1) \subset \Cee_{\beta}^*(K(u,K_1,\beta_1))$. If for example $\beta_1 = 0$ and $f$ is the symplectic potential for another K\"ahler metric $\omega' \in \scaryH$, then the assumption that $f \in \D(K_1)$ is equivalent to the assumption that $\sup_M|\varphi| < 2K$, where $\varphi$ is a normalized choice of K\"ahler potential $\omega' = \omega + i\p\bp \varphi$. We will see another interpretation of the condition \eqref{ceekbeta} in Section \ref{section-testconfig}.
    
    Second, since stability is tested on normalized potentials, the condition \eqref{Kstab-ineq-1} implies the existence of a $\lambda' > 0$ such that
    \begin{equation}\label{Kstab-ineq-3} 
        \mathcal{F}_{v,w}(f) \geq \lambda' \int_{ P} f v  dx,
    \end{equation}
    for all $f \in \Cee_\beta^*(K)$. Indeed, since $f \geq 0$ and $v$ is exponentially decaying we see that there exists $C(\gamma, v)$  such that $ v^\gamma \geq C(\gamma,v)^{-1}  v $ and consequently
   \[ \int_P u v^\gamma dx \geq C^{-1}\int_P uv dx.  \]
   Hence if \eqref{Kstab-ineq-1} holds we have that \eqref{Kstab-ineq-3} holds for $\lambda' = C^{-1}\lambda_{K,\beta,\gamma}$.
\end{remark}

Following \cite{ChenLiSheng2014}, we have the first item in Theorem \ref{Mtheorem-uniform}:
\begin{theorem}\label{existence-implies-uniform}
    Let $P$ be a Delzant polyhedron and suppose that $(v,\, w)$ are weights in the class $\mathcal{W}$. Suppose that there is an $\varepsilon \in [0,1)$ and a solution $u \in \scaryH$ to the weighted Abreu equation 
    \[ \sum_{i,j}\left( v H_{ij} \right)_{ij} = - w,\]
    where $\H = (H_{ij}) = \textnormal{Hess}^{-1}(u)$. Then $P$ is $(v,\,w)$-uniformly K-stable for all $\beta, \gamma$ with $\gamma \geq \beta$. 
\end{theorem}

\begin{proof}
 We follow the same basic strategy of \cite[Theorem 4.3]{ChenLiSheng2014}. If $P$ is not $(v,\, w)$-uniformly K-stable, then for some $K > 0$ we can find a sequence $f_k \in \Cee_\beta^*(K)$ such that 
     \[ \mathcal{F}_{v,w}(f_k) \xrightarrow{k \to \infty} 0 \hspace{.3in} \textnormal{ and } \hspace{.3in} \int_{ P} f_k v^{\gamma} dx = 1. \]  
     In particular we have a uniform local $L^1$ bound on $f_k$, which by convexity implies that, after passing to a subsequence, we may assume that there exists a continuous convex function $f \in C^0(\textnormal{int}(P))$ such that $f_k \to f$ uniformly on compact subsets of $\textnormal{int}(P)$. By definition we have
    \begin{equation*}
        \int_{P \backslash H^{\delta^*}} f_k v^{\gamma} dx \leq  \int_{P \backslash H^{\delta^*}} f_k v^\beta dx \leq K(\delta^*)^{-1}.
    \end{equation*}
    which goes to zero uniformly in $k$ as $\delta^* \to \infty$. Hence we can find a $\delta^* > 0$ such that $\int_{P \backslash H^{\delta^*}} f_k v^\gamma dx$ is as small as we like, uniformly for all $k$. Fix such a $\delta^*$. Then we also have that 
    \[ \int_P f_k |w| dx \leq C \int_P f_k v^{\beta^*}dx \leq C \int_P f_k v^\gamma dx = C \]
    as long as $\gamma \leq \beta^*$. Since $\mathcal{F}_{v,w}(f_k) \to 0$, this means that 
    \begin{equation}\label{boundarylimit} \int_{\p P} f_k  d\sigma < C. \end{equation}
    Since $f_k$ are normalized, we know that they are increasing along radial lines. Let fix $\delta_1 <<1$ sufficiently small, and let $\delta^*_1 \geq \delta^*$ be some fixed constant with the property that the radial dilation of $\p P_{\delta_1} \cap H_{\delta^*}$ out to the boundary of $P$ is contained in $\p P \cap H_{\delta^*_1}$. Then it follows that the corresponding radial dilation of $\p P_{\delta} \cap H_{\delta^*}$ is contained in $\p P \cap H_{\delta^*_1}$ for all $\delta < \delta_1$.
    Then \eqref{boundarylimit} implies 
    \[ \int_{\p (P \cap H_{\delta^*_1})} f_k d\sigma < C(v), \]
    since $v$ is bounded away from zero on $H_{\delta^*}$. From the previous discussion, we see that 
    \begin{equation*}
        \int_{P\backslash(P_{\delta} \cap H_{\delta^*})} f_k v^{\gamma} dx \leq \delta C(v,\gamma) \int_{\p (P \cap H_{\delta^*_1})} f_k d\sigma  \leq C(v) \delta.
    \end{equation*}
In particular, we can find a $\delta$ small enough such that $\int_{P\backslash(P_{\delta} \cap H_{\delta^*})} f_k v dx$ is as small as we like, uniformly in $k$, and we thus fix such a $\delta$. Finally, since $f_k \to f$ uniformly on $Q_{\delta, \delta^*}$, it follows that we can find an $N_0$ sufficiently large such that
\[\left| \int_{Q_{\delta, \delta^*}} f_k v^\gamma dx - \int_{Q_{\delta, \delta^*}} f v^\gamma dx \right|  \]
can be made arbitrarily small for all $n \geq N_0$. Hence we have 
\[ \lim_{n \to \infty} \int_{P}f_k v^\gamma dx = \int_{P} f v^\gamma dx. \]
On the other hand, by the (local) arguments of \cite[Proof of Theorem 4.3]{ChenLiSheng2014} we see that it follows from Lemma \ref{ibplemma} that $\mathcal{F}_{v,w} \to 0$ implies $f = 0$. This is a contradiction, and thus the result follows. 
\end{proof}

\subsection{Toric test configurations}\label{section-testconfig}

Suppose for the moment that $N$ is a compact toric manifold with Delzant polytope $\Delta$. Then in this case there is a well-known interpretation of the piecewise-linear convex functions considered in \ref{PL-and-stable-def} in terms of \emph{test configurations} \cite{DonStabTor} (see \cite{LahdiliWeighted} for a generalization to the compact weighted case). To see this, let $f=\max\{\ell_1, \dots, \ell_k \} \in C^0(\Delta)$ be convex piecewise-linear and $R >0$ be some constant such that $f \leq R$ on $\Delta$. Then we can define a polytope $Q_{f,R} \subset \Delta \times \R$ by 
\begin{equation}\label{testconfig-poly}
    Q_{f,R} = \left\{ (x, \, y) \in \Delta \times \R \: | \: 0 \leq y \leq R - f(x) \right\}.
\end{equation}
We assume further that $Q_{f,R}$ has rational vertices $v \in \Gamma_{\Q} \times \Q$, so that by Lemma \ref{makeavariety} there exists a compact toric variety $\mathcal{X}^f$ together with a line bundle $\mathcal{L}^f \to \mathcal{X}^f$ associated to $Q_{f,R}$. Now a general feature of toric geometry is that one can compute the sections of such line bundles in an explicit way, as follows \cite{CLS}. If we let $m = (m_1, \dots, m_{n+1}) \in (\Gamma \times \Z)  \cap Q_{f,R}$ be an integer point in $Q_{f,R}$, then one can associate a character $\chi^m:(\Cstar)^{n+1} \to \Cstar$ by setting $\chi^m(z) = z_1^{m_1}\dots z_{n+1}^{m_{n+1}}$. Then any character of this form compactifies to a well-defined section $s^m: \mathcal{X}^f \to \mathcal{L}^f$, and since here $\mathcal{X}^f$ is compact the sections $s^m$ for $m \in (\Gamma \times \Z) \cap Q_{f,R}$ in fact form a basis for $H^0(\mathcal{X}^f, \mathcal{L}^f)$. Moreover, we have a distinguished $\Cstar$-action on $\mathcal{X}^f$ coming from the extra $\R$-direction, whose weight on any section $s^m$ is equal to $m_{n+1}$. For $\bar{m} \in \Gamma \cap \Delta$ define $\bar{m}_i = (\bar{m}, i) \in (\Gamma \times \Z) \cap Q_{f,R}$. Therefore, as Donaldson shows \cite{DonStabTor}, we have a well-defined map 
\[ \pi: \mathcal{X}^f \to \P^1 \hspace{.1in} \textnormal{ given by } \hspace{.1in} \pi(p) = [s^{\bar{m}_{i}}(p): s^{\bar{m}_{i+1}}(p)]\]
for any choice of $s^{\bar{m}_{i+1}}, s^{\bar{m}_{i}}$ with $s^{\bar{m}_{i+1}}(p)$ nonvanishing. Then $\pi$ is $\Cstar$-equivaraint, has fiber $\pi^{-1}(\infty) \equiv N$, and in fact gives a test configuration in the sense of \cite{DonStabTor}. Moreover, the Donaldson-Futaki invariant of $(\mathcal{X}^f, \mathcal{L}^f)$ is up to a fixed multiplicative factor equal to \cite[Proposition 4.2.1]{DonStabTor}
\[\mathcal{F}_{1,1}(f) = 2\int_{\p \Delta} f d\sigma - \int_{\Delta}f dx. \]
In \cite{LahdiliWeighted}, this was generalized to the compact weighted case, to show that there is a Donaldson-Futaki invariant associated to the weighted problem given the same data $(\mathcal{X}^f, \mathcal{L}^f)$ which coincides with $\mathcal{F}_{v,w}(f)$. 

In the non-compact setting, the construction \eqref{testconfig-poly} will not in general give us a test configuration, as we clearly will not always have that $f$ is bounded above. Hence we have 
\begin{definition}\label{admissiblePL}
    Let $M$ be a smooth quasiprojective toric variety with polyhedron $P$. We say that a convex piecewise-linear function $f = \max\{\ell_1, \dots, \ell_k\}$ on $P$ is \emph{admissible} if each $\ell_j(x) = \langle x, \, b_j \rangle + a_j$ has the property that $b_j \in -C^*(P)$.
\end{definition} 
It is not hard to see that $f$ is eventually negative on $P$ if and only if it is admissible. Given an admissible $f$ on $P$, we have therefore that $f$ is bounded from above and hence for a suitable choice of $R$ we can define $Q^f \subset P \times \R$ exactly as in \eqref{testconfig-poly}. Then just as above the polarized toric variety $(\mathcal{X}^f, \, \mathcal{L}^f)$ admits a $\Cstar$-equivariant map $\pi: \mathcal{X}^f \to \P^1$ sharing the same properties as above. For our present purposes, we simply define the $(v,\, w)$ Donaldson-Futaki invariant of $(\mathcal{X}^f, \, \mathcal{L}^f)$ to be a fixed multiple of $\mathcal{F}_{v,w}(f)$.

Before moving on, we give an example that illustrates some interesting phenomena. Let $B = B_{1} \times \dots \times B_k$ be a compact K\"ahler manifold endowed with a product metric $\omega_B = \sum_{a=1}^\ell \omega_a$, and suppose that $\omega_a$ are individually cscK and that $[\omega_{B_a}] \in H^2(B_a, 2\pi \Z)$ are integral. In particular for $p_a \in \Z$ there exists a line bundle $L \to B$ whose curvature form $\theta$ satisfies
\begin{equation*}
    d\theta = \sum_{a=1}^k p_a  \pi_{B_a}^{*}\omega_{B_a}
\end{equation*}
when pulled back to the total space $Y$ of $L$. Set $P = [-1, \infty)$ and suppose we have $(v,\,w) \in \mathcal{W}(P)$. Then we seek $(v,\ w)$-cscK metric on the total space $Y$ of a special form. We wish to find a function $\Theta:[-1, \infty) \to \R$ such that the data 
\begin{equation}\label{calabiansatz-cbundles}
\begin{split}
    & g_Y = \sum_{a=1}^k (p_a x + c_a) \pi_{B_a}^{*}g_{B_a} + \frac{dx^2}{\Theta(x)} + \Theta(x) \theta^2,\\
    &\omega_Y = \sum_{a=1}^k (p_a x + c_a) \pi_{B_a}^{*}\omega_{B_a} + dx \wedge \theta,
\end{split}
\end{equation}
solves the $(v,\,w)$ cscK-equation on $Y$. This is another formulation of the Calabi Ansatz \ref{calabiansatz} which we will meet again in more detail in Section \ref{section-chili}. This same problem was studied in the projective bundles case in \cite{LahdiliWeighted, AMT-projbundles} (see also \cite{Sz06}). We defer to Section \ref{section-ssfibrations} for details, but for now it suffices to note that the function $x:Y \to \R$ is a moment map for the $S^1$-action on the fibers of $L \to B$ whose image is equal to $P = [-1, \infty)$. In this case the data $g_{\Theta} = \frac{dx^2}{\Theta(x)} + \Theta(x) \theta^2,  \, \, \omega_\theta = dx \wedge \theta$ corresponds to an AK metric $\omega$ on $\C$ (with the standard complex structure) if and only if the function $\Theta$ satisfies certain boundary conditions at $x = -1$, and the function $\Theta$ is precisely the data $\H = H_{11}$ of \eqref{Hdef} associated to $\omega$. As we will see in Section \ref{section-ssfibrations}, this is special case of the \emph{semisimple principal fibration} construction. In particular, by Lemma \ref{fibrationweights} below (see also \cite{ApJuLa, LahdiliWeighted}), we in fact have that a metric of the form \eqref{calabiansatz-cbundles} on $Y$ is $(v,\, w)$-cscK if and only if the metric $\omega$ on $\C$ which is $(\tilde{v}, \, \tilde{w})$-cscK, where
\begin{equation}\label{cbundles-newweights}
    p(x) = \Pi_{a=1}^k(p_a x + c_a)^{n_a}, \hspace{.2in} \tilde{v} = pv, \hspace{.2in} \tilde{w} = p\left(w - v\sum_{a = 1}^k \frac{s_a}{p_a x + c_a}\right),
\end{equation}
where $s_a := \Scal(\omega_{B_a})$ are constant by assumption. Clearly we have $(\tilde{v},\, \tilde{w}) \in \mathcal{W}$. Then the $(\tilde{v}, \, \tilde{w})$-cscK equation becomes (compare \eqref{genAbreu})
\begin{equation}\label{profileODE}
    \left(\tilde{v} \Theta \right)''(x) = -\tilde{w} = p\left(v\sum_{a = 1}^k \frac{s_a}{p_a x + c_a}  - w\right).
\end{equation}
Then the conditions we require on $\Theta$ in order to ensure that the metric $\omega$ is well-defined are
\begin{equation}\label{profileboundary}
    \Theta(-1) = 0, \hspace{.3in} \Theta'(-1) = 2,
\end{equation}
and 
\begin{equation}\label{profilepositive}
    \Theta > 0.
\end{equation}
We can compute directly as in \cite{AMT-projbundles} that the unique solution to \eqref{profileODE} with initial conditions \eqref{profileboundary} is given by 
\begin{equation}\label{profilesolution}
    \tilde{v}\Theta(x) = 2\tilde{v}(-1)(1 + x) - \int_{-1}^x(x-t)\tilde{w}(t) dt.
\end{equation}
In order for a solution to exist then we only need to verify that $\Theta(x) > 0$. Suppose then that $P$ is $(\tilde{v}, \, \tilde{w})$ K-stable. Then following \cite{LahdiliWeighted}, for any given point $x_0 \in (-1, \infty)$ we set 
\begin{equation}\label{simplePL}
    f_{x_0}(x) = \max\{ x_0 - x , \, 0 \}.
\end{equation} 
Integrating by parts, it follows just as in \eqref{piecewiseibp} that 
\begin{equation*}
    \mathcal{F}_{\tilde{v}, \tilde{w}}(f_{x_0}) = p(x_0)v(x_0)\Theta(x_0).
\end{equation*}
Note that the integration by parts is justified with no assumptions on $\Theta$ in this case, as $f_{x_0}$ is identically zero outside of the compact set $K_{x_0} = [-1,\, x_0]$. Therefore, in combination with Proposition \ref{existenceimpliesstable} we have:
\begin{prop}\label{Kstableimpliesexistence-cbundles}
    There exists a $(\tilde{v}, \tilde{w})$-cscK metric on $\C$, and consequently a $(v,\, w)$-cscK metric on $Y$ of the form \eqref{calabiansatz-cbundles} if $P = [-1, \infty)$ is $(\tilde{v},\, \tilde{w})$ K-stable.
\end{prop}

Notice that, at least in order to obtain our formal solution $\Theta$, we did not need to assume that the weighted Futaki invariant $\mathcal{F}_{\tilde{v}, \tilde{w}}$ vanishes on the affine-linear functions. Next we show that, if we do indeed include this assumption, then we can make connections with the asymptotic conditions of Definition \ref{asymptotic-metric-general}, at least for certain choices of weights $(\tilde{v}, \, \tilde{w})$. 
\begin{lemma}
    Let $P = [-1, \infty)$,  $(\tilde{v}, \, \tilde{w}) \in \mathcal{W}(P)$, and $\Theta$ be the formal solution to \eqref{profileODE} given by \eqref{profilesolution}. Then, if $\mathcal{F}_{\tilde{v}, \tilde{w}}$ vanishes on the affine-linear functions, we have that 
    \begin{equation}\label{profile-futakivanishes}
        \left(\tilde{v}\Theta\right)(x) = - \int_{x}^\infty (t - x) \tilde{w}(t) dt.
    \end{equation}
\end{lemma}
\begin{proof}
    It is straightforward to compute using $\mathcal{F}_{\tilde{v}, \tilde{w}}(1) = \mathcal{F}_{\tilde{v}, \tilde{w}}(t) = 0$ that 
    \begin{equation*}
        \int_{-1}^\infty(x-t) \tilde{w}(t) dt = 2\tilde{v}(-1)(1 + x).
    \end{equation*}
    Combining this with \eqref{profilesolution} gives \eqref{profile-futakivanishes}.
\end{proof}
Using this, we see that, under certain assumptions on the weights, our solution $\Theta$ will give rise to a metric on $\C$ which satisfies the conditions of Definition \ref{asymptotic-metric-general}. 
\begin{prop}\label{mainprop2-inthebody}
    Suppose that the weights satisfy 
    \[ \tilde{v}(x) = p(x) e^{-\lambda x}, \hspace{.3in} \tilde{w}(x) = q(x) e^{-\lambda x},\]
    where $p(x)$ is a positive rational function on $P$, $q(x)$ is a polynomial, and $\lambda > 0$. Suppose that $P$ is $(\tilde{v}, \, \tilde{w})$ K-stable, so that by Proposition \ref{Kstableimpliesexistence-cbundles}, there is a $(\tilde{v}, \, \tilde{w})$-cscK metric $\omega$ on $\C$ whose data $\H$ as in \eqref{Hdef} is precisely the solution $\Theta$ to \eqref{profileODE}. Then $\omega \in \scaryH$ for any $\varepsilon > 0$.
\end{prop}
\begin{proof}
    Using the identity 
     \[ \int f(x) e^{-\lambda x} dx = \sum_{i = 0}^\infty \frac{f^{(n)}(x)}{\lambda^{n+1}} e^{-\lambda x}, \] 
      it follows from \eqref{profile-futakivanishes} together with our assumptions on the weights that both $\Theta$ and $\Theta'$ (and in fact all derivatives), are rational functions in $x$. Conditions \ref{asymptotics1} and \ref{asymptotics2} of Definition \ref{asymptotic-metric-general} follow immediately from this together with Lemma \ref{gradexpressionlemma}. Condition \ref{asymptotics3} is similar. Indeed, let $u$ be a symplecitc potential for $\omega$ on $P$ so that $(u'')^{-1} = \H = \Theta$. It follows that $u''$ itself is rational in $x$, so that by the mean value theorem there exists an $a \in \R$ and a $C > 0$ such that 
      \[ u(x) \leq C(|x|^a + 1).  \]
      Condition \ref{asymptotics4} is immediate from the boundary conditions \ref{boundaryconditions}. 
\end{proof}
A noteworthy feature of the above construction is that in order to test for existence, in this case it suffices to check on test configurations generated by piecewise-linear functions of the form \eqref{simplePL}. In particular, we only need to test on a family of functions whose slope at infinity is \emph{uniformly bounded}. It is straightforward to see that for this family $f_{x_0}$, the uniform condition \eqref{ceekbeta} is satisfied for any choice of $K>0$ sufficiently large. With this in mind, we define 
\begin{definition}\label{D-admissible}
    Let $D > 0$ be a fixed constant. Let $f = \max\{\ell_1, \dots, \ell_k\}$ be an admissible convex piecewise-linear function, so that we can decompose $P = \cup \Delta_i$ where $f(x) = \ell_i(x) = \langle x, \, b_i \rangle + a_i$ on $\Delta_i$. Suppose that $0 \in \overline{\Delta}_{i_1} \cap \dots \cap \overline{\Delta}_{i_r}$ We say that $f$ is $D$-admissible if there exists $j \in \{i_1, \dots, i_r\}$ such that 
    \[ |b_j| < D, \hspace{.1in} \textnormal{ and } \hspace{.1in} f(x) \leq a_j + D \hspace{.1in} \textnormal{ for all } \hspace{.1in} x \in \{ \langle x, \, - b_j \rangle \geq 1 \}.\]
    We say that a toric test configuration $(\mathcal{X}^{f}, \, \mathcal{L}^f)$ is $D$-admissible if it is generated by a $D$-admissible function $f$.
\end{definition}
Then for any $D$-admissible $f$, set $b_+ = - b_j$. Then the set $H_1 = \{\langle x, \, b_+ \rangle \leq 1 \}$ is bounded in $P$ as $b_+ \in C^*(P)$. Then if we set 
\begin{equation}\label{PLchange}
    f_+(x) := f(x) - \ell_j(x), 
\end{equation}
we will have that $f_+(0) = 0$ and 
\[ f_+(x) \leq \langle x , \, b_+ \rangle + D \leq D \left( |x| + 1 \right) \]
for all $x \in P\backslash H_1$. Suppose then that we are given $(v, \, w) \in \mathcal{W}(P)$. Then clearly we have that $f_+ \in \Cee_\beta(K)$ for some fixed $K = K(D)$, independent of our original $f$ (note that the assignment $f \mapsto f_+$ is not unique, but any such $f_+$ will be in $\Cee_\beta(K)$). In particular, we can view Definition \ref{Kstab-def} as a version of uniform K-stability whose test configurations are $D$-admissible. 

Taking the example of \eqref{simplePL}, we claim that $f_{x_0}=  \max\{x_0 - x, \, 0\}$ is $1$-admissible. Indeed, we can take $\ell_j(x) = x_0 - x$, $b_j = -1$, $a_j = x_0$. Then since globally $f_{x_0}(x) \leq f_{x_0}(0) = x_0 + 1$, we see that $f_{x_0}$ satisfies Definition \ref{D-admissible} with $D = 1$. Therefore we can interpret Proposition \ref{Kstableimpliesexistence-cbundles}:
\begin{corollary}
    In the line bundles problem \eqref{calabiansatz-cbundles} above, there exists an AK $(\tilde{v}, \, \tilde{w})$-cscK metric on $\C$ with moment image equal to $P = [-1, \infty)$ if and only if $\mathcal{F}_{\tilde{v}, \tilde{w}}$ is strictly positive on $1$-admissible test configurations.
\end{corollary}
\begin{remark}
    At least for test configurations $\mathcal{X}^{x_0}$ generated by piecewise linear functions of the form \eqref{simplePL} (viewed as a test configuration for $Y$), it is possible to interpret $\mathcal{F}_{\tilde{v}, \tilde{w}}(f_{x_0})$ as a Futaki invariant associated directly to $\mathcal{X}^{x_0}$. Indeed, in this setting we have that $\mathcal{X}^{x_0}$ is smooth. Then for an appropriate choice of K\"ahler metric $\Omega$ on $\mathcal{X}$, one can define
    \[ \mathcal{F}_{v,w}(\mathcal{X}^{x_0}) = \int_{\mathcal{X}^{x_0}}\left(\Scalv(\Omega) - w(\mu_{\Omega})\right) \Omega^{n+1} - 8\pi \int_Y v(\mu_{\omega_Y}) \omega_Y^n.\] 
    By the arguments of \cite[Sections 9, 10]{LahdiliWeighted}, it follows that there is a universal constant $ C(B_a, [\omega_{B_a}])$ such that $C(B_a, [\omega_{B_a}])\mathcal{F}_{v,w}(\mathcal{X}^{x_0}) = \mathcal{F}_{\tilde{v}, \tilde{w}}(f_{x_0})$. In this way Proposition \ref{Kstableimpliesexistence-cbundles} can be interpreted in terms of stability criteria directly on $Y$ as is done in \cite{LahdiliWeighted}.
\end{remark}

\subsubsection{Proof of Corollary \ref{Mcorollary2}} 

Let $P = [-1, \, \infty)$ and suppose that $\tv$ is a positive weight function on $P$, and let $\tw = 2\left(v + x \frac{d\tv}{dx}\right) $ be the weight corresponding to the $v$-soliton equation via Lemma \ref{vsolitonw}. Suppose that $(\tv, \, \tw) \in \mathcal{W}(P)$. The next lemma says that the vanishing of the weighted Futaki invariant associated to the $v$-soliton problem (see Section 2.4) is sufficient in this special case to deduce K-stability of $P$ with respect to $(\tv, \, \tw)$. 
\begin{lemma}\label{1Dvsoliton}
If the weight $\tv$ satisfies 
\[ \int_{-1}^\infty x \tv \, dx = 0, \]
then $P$ is $(\tv, \, \tw)$ K-stable. 
\end{lemma} 

\begin{proof}
It's straightforward to see, using the fact that $\tw = 2\left(\tv + x \frac{d\tv}{dx}\right) = 2d(x\tv)$, that 
\[ \mathcal{F}_{\tv, \, \tw}(1) = -2\tv(-1) - \int_{-1}^\infty \tw \, dx = 0.   \] 
Note that the minus sign on the boundary term comes from the interpretation of the measure $d\sigma$ in this case. Together with Lemma \ref{vsolitonfutakiinvariantlemma}, it follows that $\mathcal{F}_{\tv, \, \tw}$ vanishes on the affine linear functions. 

Now let $f$ be any piecewise linear convex function. Since $\mathcal{F}_{\tv, \, \tw}$ vanishes on the affine linear functions, we can assume without loss of generality that $f \geq 0$ and is strictly positive for all $x \in P$ sufficiently large. Now suppose that $f$ is equal to the affine linear function $\ell_j$ on $[b_j, \, b_{j+1}]$ for $j = 1, \dots, N-1$ and $\ell_N$ on $[b_N, \, \infty)$. Then, applying Lemma \ref{vsolitonfutakiinvariantlemma} on each interval gives us 
\begin{equation*}
\begin{split}
	\mathcal{F}_{\tv, \, \tw}(f)  &= \int_{\p P} f \tv \, d\sigma  - \int_P f\tw \, dx \\
			&= - 2f(-1)\tv(-1) - \sum_{j=1}^{N-1}\int_{b_j}^{b_{j+1}} \ell_j \tw \, dx  - \int_{b_N}^\infty \ell_N \tw \, dx \\
			&= - 2f(-1)\tv(-1) - \sum_{j=1}^{N-1}\left[2 f\tv \, \bigg|_{b_{j}}^{b_{j+1}}- 2\int_{b_j}^{b_{j+1}} \ell_j \tv \, dx \right] +  2f(b_{N})\tv(b_{N}) + 2\int_{b_N}^\infty \ell_N \tv \, dx \\
			&= 2 \int_{-1}^\infty f \tv \, dx -2f(-1)\tv(-1) - 2\sum_{j=1}^{N-1}\left[f(b_{j+1})\tv(b_{j+1}) - f(b_{j})\tv(b_{j})\right] + 2f(b_{N})\tv(b_{N})  \\
			&= 2 \int_{-1}^\infty f \tv \, dx. 
\end{split}
\end{equation*}
 Since $f$ and $\tv$ are both positive, it follows that $\mathcal{F}_{\tv, \, \tw}(f) > 0$. 
\end{proof}

Let $(B, \, L)$ be as in the statement of Corollary \ref{Mcorollary2}, i.e. $B$ is a K\"ahler-Einstein Fano manifold of dimension $n$ and $L \to B$ is a negative line bundle such that $L^{\frac{1}{\kappa}} = K_B$, $0 < \kappa < 1$. Let $\omega_B$ be the K\"ahler-Einstein metric on $B$ with $\omega_B = \tau \Ric_B$, normalized such that 
\[ \tau = \frac{1}{1 - \kappa}. \]
Note that this condition implies that $\kappa x + \tau^{-1} > 1$ on $P = [-1, \, \infty)$. By Lemma \ref{fibrationsolitons} below, if $\C$ admits a $v$-soliton with weight 
\[ \tv(x) = (\kappa x + \tau^{-1})^n e^{-\lambda x} \]
for some $\lambda > 0$, then the total space of $L \to B$ will admit a shrinking K\"ahler-Ricci soliton with respect to a multiple of the radial scaling vector field on the fibers of $L$. By Proposition \ref{Kstableimpliesexistence-cbundles} and Lemma \ref{1Dvsoliton}, $\C$ admits such a $\tv$-soliton as long as $\tv$ satisfies $\int_{-1}^\infty x \tv \, dx = 0$. It is then straightforward to see that we can always choose a $\lambda > 0$ such that $\tv$ has this property. Indeed, the function 
\[ \lambda \mapsto \int_{-1}^\infty  (\kappa x + \tau^{-1})^n e^{-\lambda x} \, dx  \]
is convex, and can be seen to be proper by an argument similar to that in \cite[Proposition 3.1]{uniqueness}. Therefore it has a critical point which is precisely the value of $\lambda$ we need. This completes the proof of Corollary \ref{Mcorollary2}. 

\subsection{The weighted Mabuchi energy}

Fix a background AK metric $\omega_0 \in \scaryH$ on $M$ with associated symplectic potential $u_0 \in C^\infty(P)$. As in \cite{LahdiliWeighted, simonYTD}, we can define formally the weighted Mabuchi energy by:
\begin{definition}
The weigted Maubchi functional is defined for $u \in \Cee_\beta$ as 
\begin{equation}\label{mabuchidef}
        \M_{v,w}(u) = \mathcal{F}_{v,w}(u) - \int_P \log\det\left((u_0)^{ik}u_{kj} \right)v dx.
\end{equation}
\end{definition}

\begin{lemma}[{c.f. \cite[Proposition 7.7]{simonYTD}}]\label{mabuchi-welldefined}
    The Mabuchi energy $\M_{v,w}$ is well-defined on $\Cee_\beta$ as a functional taking values in $(-\infty, \infty]$.
\end{lemma}

\begin{proof}
    We have already seen that $\mathcal{F}_{v,w}$ is well-defined on $\Cee_\beta$. For the second term, we first let $u = u_0$ be the symplectic potential of our fixed background AK metric $\omega_0 \in \scaryH$. By Lemma \ref{spacescompare-1} we have $ u_0 \in \Cee_\beta$, so that clearly $\M_{v,w}(u_0) $ is well-defined and finite. Then for any $f \in \Cee_\beta$, \ref{mabuchi-welldefined} follows from the convexity of $-\log\det$ and Lemma \ref{ibplemma} applied to the difference $f - u_0$ by a standard argument, see \cite[Corollary 3.3.10]{DonStabTor}, \cite[Proposition 7.7]{simonYTD}.
\end{proof}

\begin{prop}\label{mabuchi-boundedbelow}
     Suppose that $(v, \, w) \in \mathcal{W}$, $u \in \scaryH$ is a solution to \eqref{genAbreu}, and $u_0$ is a background metric with the property that the Mabuchi energy defined relative to $u_0$ satisfies $\M_{v,w}(u) < \infty$. Then the weighted Mabuchi energy defined relative to $u_0$ is bounded from below on $\Cee_\beta$. 
\end{prop}

\begin{proof}
    To begin we observe that $\Cee_\beta$ is convex. That is, if $f_1, f_2 \in \Cee_\beta$, then it is clear that $f_t = tf_1 + (1-t)f_2 \in \Cee_\beta$. Since $\mathcal{F}_{v,w}$ is linear, then it follows from the convexity of $-\log\det$ that $\M_{v,w}$ is convex on $\Cee_\beta$. To see morover that there exists at least one choice of $u_0$ with the required properties, simply take $u_0 = u$. Clearly then we have $\M_{v,w}(u) = \mathcal{F}_{v,w}(u) - \log(n)\int_{P}v dx < \infty.$

    We now investigate the differentiability of $\M_{v,w}$ in $\Cee_\beta$ at the distinguished point $u$ corresponding to our solution of \eqref{genAbreu}. Now let $f \in \Cee_\beta$ and set $u_t = (1-t)u + t f$, $\H_t = \textnormal{Hess}^{-1} u_t$. Then simply from the convexity of $f$ we have
    \begin{equation}\label{hessian-(1-t)}
        (u_t)_{ij} \geq (1-t) u_{ij}. 
    \end{equation}
By comparing the eigenvalues of $\textnormal{Hess}(u_t)$ and $\textnormal{Hess}(u)$ we see that therefore 
\begin{equation*}
    \H_t \leq \frac{1}{1-t}\H
\end{equation*}
If we write $g_t(x) = -\log\det (u_t)_{ij}$, then on $P_\delta$ we have, for $t < \frac{1}{2}$ say
\begin{equation*}
    \begin{split}
    \left| \ddt (g_t v) \right| &= \left|\sum_{i,j}v (H_t)_{ij}(u - f)_{ij} \right| \\
    &\leq \frac{1}{1-t} \left|\sum_{i,j}vH_{ij}u_{ij}\right| + \frac{1}{1-t}\left|\sum_{i,j}vH_{ij}f_{ij}\right| \\
     &\leq 2nv + 2\sum_{ij}vH_{ij}f_{ij},
    \end{split}
\end{equation*}
noting that both terms in the last inequality are positive. By Lemma \ref{ibplemma}, we have that the right hand side is a fixed integrable function on $P$ independent of $t$, hence by another application of Lemma \ref{ibplemma}, 
\begin{equation*}
\begin{split}
    \left.\ddt \right|_{t = 0} \int_{P} g_t v dx &= \int_{P} \sum_{i,j} vH_{ij}(u - f)_{ij} dx \\
        &= \int_{P} \sum_{i,j} vH_{ij} u_{ij} dx - \int_{P} \sum_{i,j} vH_{ij} f_{ij} dx \\
    & =  \mathcal{F}_{v,w}(u) - \mathcal{F}_{v,w}(f) = \mathcal{F}_{v,w}(u-f).
\end{split}
\end{equation*}
The linear part of $\M_{v,w}$ is clearly differentiable along this path, and hence we see that 
\begin{equation*}
    \left.\frac{\p}{\p t}\right|_{t = 0} \M_{v,w}(u_t) = 0.
\end{equation*}
Since $f \in \Cee_\beta$ was arbitrary it follows from the convexity that $\M_{v,w}$ attains its minimum on $\Cee_\beta$ at $u$.
\end{proof}
With this in place, we can now prove Corollary \ref{Mcorollary}.
\begin{proof}[Proof of Corollary \ref{Mcorollary}]
     We use the general strategy of \cite{GuanUniqueness} (see \cite[Theorem 3.1]{vestisnotes}). Suppose that $\varepsilon \in \left[ 0, \max\{ \beta^*, \frac{1}{2}\} \right)$, and that $\omega_1, \omega_2 \in \scaryH$ satisfy 
    \[ \Scalv(\omega_1) = \Scalv(\omega_2) = w. \]
    As we have seen, for the appropriate normalization the moment maps $\mu_1, \mu_2$ associated to $\omega_1, \omega_2$ will have common image equal to a Delzant polyhedron $P \subset \t^*$. We thus have two symplectic potentials $u_1, u_2 \in C^\infty(P) \cap C^0(\overline{P})$ which satisfy \eqref{genAbreu}. By Lemma \ref{spacescompare-1} it follows that $u_i \in \Cee_\beta$ for any $\beta \geq \varepsilon$. We then define the Mabuchi energy $\M_{v,w}$ \eqref{mabuchidef} on $\Cee_\beta$ relative to $u_1$, and we claim that 
    \begin{equation}\label{usethis2} \M_{v,w}(u_2) = \mathcal{F}_{v,w}(u_2) - \int_P \log\left( (u_1)^{ik}(u_2)_{kj}\right)v dx < \infty.\end{equation}
    Indeed as $u_2 \in \Cee_\beta$, we have $\mathcal{F}_{v,w}(u_2)< \infty.$ For the second term we use the condition that $\omega_1, \omega_2$ are uniformly equivalent, which in this case implies that 
    \begin{equation} \label{usethis3} C^{-1} \Id \leq \textnormal{Hess}^{-1}(u_1)\textnormal{Hess}(u_2)  \leq C \Id, \end{equation}
    which gives \eqref{usethis2}. Set $u_t = tu_2 + (1-t)u_1$. Then by Proposition \ref{mabuchi-boundedbelow}, we have that the function $\M_{v,w}(t)$ given by
    \[ t \mapsto \M_{v,w}(u_t)\]
    is convex and attains its minimum value at the two endpoints $t = 0, 1$, hence is constant. We have already seen that $\M_{v,w}(t)$ is differentiable in $t$ for any $t < 1$ with first variation 
    \[ \M'_{v,w}(t) = \ddt \M_{v,w}(u_t) = \mathcal{F}_{v,w}(f) - \int_P \sum_{i,j} vH^t_{ij} f_{ij} dx,\]
    where $f = u_2 - u_1$. We would like to show that in fact $\M_{v,w}$ is twice differentiable at $t = 0$. To do this, we compute 
    \begin{equation*}
    \begin{split} 
    \left| \ddt \left( -\sum_{i,j} vH^t_{ij} f_{ij} \right) \right|&= v  \tr\left(\left[ \H^t \cdot \textnormal{Hess}(f) \right]^2\right)  \\
        & \leq v  \left[\tr\left( \H^t \cdot \textnormal{Hess}(f) \right)\right]^2  \\
        &= v  \left[\tr( \H^t \cdot \textnormal{Hess}(u_2)) - \tr(\H^t \cdot \textnormal{Hess}(u_1))  \right]^2  \\
        & \leq v \left(\frac{1}{1-t} \right)^2  \left[ \tr(\H^1\cdot \textnormal{Hess}(u_2))^2 + 2n \tr(\H^1\cdot \textnormal{Hess}(u_2))^2  + n^2  \right] \\
        & \leq C \left(\frac{1}{1-t}\right)^2 v,
    \end{split}
    \end{equation*}
where the second to last line follows from the fact that $\H^t \leq \frac{1}{1-t} \H^1$ as in Proposition \ref{mabuchi-boundedbelow}, and the last line follows once again using \eqref{usethis3}. It follows that in fact $\M_{v,w}(t)$ is twice differentiable for $t < 1$, so that in particular we have 
\[ \M''_{v,w}(0) = \left.\frac{\p^2}{\p t^2}\right|_{t=0}\M_{v,w}(u_t) = \int_P   \tr\left(\left[ \H^1 \cdot \textnormal{Hess}(u_2 - u_1) \right]^2\right) v dx, \]
from which it follows that $(u_2 - u_1)_{ij} \equiv 0$. Therefore $u_2 = u_1 + \ell$ for an affine-linear $\ell$. Translating back to the complex picture, this implies that the automorphism $A \in i\t \subset \Cstarn \subset \textnormal{Aut}(M,J)$ determined by the linear part of $\ell$ has the property that $A^*\omega' = \omega$.
\end{proof}

\begin{theorem}[{c.f. \cite[Proposition 7.9]{simonYTD}}]\label{stability-implies-proper}
    Suppose that $(v, \, w) \in \mathcal{W}$ and that $u_0 \in \scaryH$ satisfies $||\textnormal{Scal}_v(u_0)||_{C^k_{\gamma}(\overline{P})} \leq C(k) $ and
    \begin{equation}\label{vscal-a}
        \left| \textnormal{Scal}_v(u_0) - w \right| < C v^\gamma
    \end{equation}
    for $\gamma \geq \beta$. Then if $\D \subset \Cee_\beta^*$ is any convex subset with the property that there exists a $\lambda_\D > 0$ such that  
    \begin{equation}\label{dkstable}
        \mathcal{F}_{v,w}(f) \geq \lambda_\D \int_{P} f v^{\gamma} dx,
    \end{equation}
    then there exists a constant $C_\D > 0$ such that 
    \begin{equation}\label{mabuchiproper-a}
        \M_{v,w}(u) \geq C_\D^{-1}\int_P u v^\gamma dx - C_\D
    \end{equation}
    for all $u \in \D$. In particular, if $P$ is $(v,\, w)$-uniformly K-stable with parameters $(\beta, \gamma)$, then there exist constants $C_{K,\beta,\gamma}$ such that 
    \begin{equation*}
        \M_{v,w}(u) \geq C_{K,\beta,\gamma}^{-1}\int_P u v^\gamma dx - C_{K,\beta,\gamma}
    \end{equation*}
    for all $u \in \Cee_\beta^*(K)$.
\end{theorem}
\begin{proof}
    Let $w_0 = \textnormal{Scal}_v(u_0)$. Then by \eqref{vscal-a}, we have that, for any $u \in \Cee_\beta^*$
    \begin{equation*}
        \left| \mathcal{F}_{v, w_0}(u) - \mathcal{F}_{v, w}(u) \right| \leq C \int_P f v^\gamma dx. 
    \end{equation*}
    By \eqref{dkstable}, it follows that 
    \begin{equation*}
        C \int_{P} f v^\gamma dx = 2C \int_{P} f v^\gamma dx - C \int_{P} f v^\gamma dx  \leq 2C\lambda^{-1}\mathcal{F}_{v,w}(f) - C\int_{P} f v^\gamma dx .
    \end{equation*}
    In particular, there exists a constant $C_1$ independent of $u$ such that 
    \begin{equation*}
        \mathcal{F}_{v, w_0}(u) \leq C_1\mathcal{F}_{v, w}(u) - C\int_{P} f v^\gamma dx. 
    \end{equation*}
    Therefore 
    \begin{equation*}
    \begin{split}
        \M_{v,w}(u) &= \mathcal{F}_{v,w}(u) - \int_P \log\det\left((u_0)^{ij}u_{ij} \right)v dx  \\
        &\geq  C_2 \mathcal{F}_{v, w_0}(u) + C_3\int_{P} f v^\gamma dx  - \int_P \log\det\left((u_0)^{ij}u_{ij} \right)v dx \\
        & = \mathcal{F}_{v, w_0}(C_2u) - \int_P \log\det\left((u_0)^{ij}(C_2 u)_{ij} \right)v dx + n\log(C_2) \int_{P}v dx + C_3\int_{P} f v^\gamma dx \\
        &= \M_{v,w_0}(C_2 u)  + n\log(C_2) \int_{P}v dx + C_3\int_{P} f v^\gamma dx.
    \end{split}
    \end{equation*}
    By construction, the background potential $u_0$ satisfies $\Scalv(u_0) = w_0$. Hence by Proposition \ref{mabuchi-boundedbelow}, we know that $\M_{v,w_0} \geq - C$ on $\Cee_\beta$ (note that $C_2 u \in \Cee_\beta^*$). Rearranging gives \eqref{mabuchiproper-a}.
    \end{proof}

This gives the following corollary, which is the second item in Theorem \ref{Mtheorem-uniform}:
\begin{corollary}\label{stability-impliesproper-cor}
Let $P$ be a Delzant polyhedron and $(v,\, w) \in \mathcal{W}$. If $P$ is $\beta$-uniformly K-stable, then there are constants $C_K$ such that
    \begin{equation}\label{mabuchiproper-beta}
        \M_{v,w}(u) \geq C_K^{-1}\int_P u v^\beta dx - C_K
    \end{equation}
    for all $u \in \Cee_\beta^*(K)$. In particular, by remark \ref{remark-normcompare}, there exist $C'_K$ such that
    \begin{equation}\label{mabuchiproper-1}
        \M_{v,w}(u) \geq (C'_K)^{-1}\int_P u v dx - C'_K.
    \end{equation}
\end{corollary}

\begin{proof}
We claim that the Guillemin potential $u_P$ \eqref{guilleminpotential} satisfies $||\textnormal{Scal}_v(u_P)||_{C^k_{a}(\overline{P})} \leq C(k)$ and \eqref{vscal-a} for any $a < \beta^*$. Indeed, since $\H_P$ is homogeneous of degree $1$ on $P$, it follows that 
    \begin{equation*}
        \left|\p_{\alpha}\textnormal{Scal}_v(u_P) \right| \leq \sum_{|\alpha'| = |\alpha| + 2} \sum_{ij} \left|\p_{\alpha'}(v H_{ij})\right| \leq C(|\alpha|) (|x| + 1) v^{\beta^*} \leq C(|\alpha|) v^{a}.
    \end{equation*}
    Note that for $v$ exponentially decaying with derivatives, it is clear that $u_P \in \scaryH$. The fact that $u_P$ satisfies \eqref{vscal-a} follows from the above for $|\alpha| = 0$ together with the fact that $|w| \leq C v^{\beta^*}$ by the fact that $(v,\,w) \in \mathcal{W}$.
\end{proof}


As is clear from the proof, the condition that $\omega$ and $\omega'$ be uniformly equivalent could be replaced by a weaker condition using the weight $v$. For example, the same result is true if we only demand that that there exists $C > 0, k > 0$ such that 
\[ C^{-1}v(\mu_{\omega})^{2k} \omega' < v(\mu_{\omega})^k \omega < C \omega'.\] 

\section{Product models }\label{section-productmodels}

\subsection{Definitions and asymptotic geometry}

Let $P \subset \t^*$ be a Delzant polyhedron defined by \eqref{polydef}, and let $C = C(P)$ be its corresponding recession cone.  Suppose that $C$ has dimension equal to $k\in 1, \dots, n$ (of course we could also consider the case where $k = 0$,  but we will primarily be focused on the non-compact situation in this paper).  Let $\t_C^* \subset \t^*$ denote the smallest linear subspace of $\t^*$ containing $C$, so that the interior of $C$ is open in $\t_C^*$.  Let $\t_V^* = \left( \t_C^*\right)^\perp \subset \t^*$.  By the definition of $C$ \eqref{recessioncone}, it follows that we can decompose the set $I = \{1, \dots, N\} = I_1 \cup I_2 \cup I_3$, where 
\begin{equation*}
\begin{array}{lcl}
 i \in I_1 & \Rightarrow &  \nu_i \in \t_V^*, \, \textnormal{i.e. } \langle \nu_i , \, c \rangle = 0, \textnormal{ for all } c \in C  \\
 i \in I_2 & \Rightarrow & \langle \nu_i , \, c \rangle > 0 \textnormal { for all } c \in C \\
 i \in I_3 & \Rightarrow & \textnormal{there exists } c_1, c_2 \in C, \, \langle c_1 , \, \nu_i \rangle = 0, \,  \langle c_2 , \, \nu_i \rangle > 0,
\end{array}
\end{equation*}
and moreover that the collection $\{ \nu_i \: | \: i \in I_1 \}$ spans $\t_V^*$. Clearly a facet $F$ of $P$ is compact if and only if its inner normal $\nu_i$ satisfies $i \in I_2$.

\begin{definition}\label{asymptotic-complexstructure}
We say that $M = M_P$ is a \emph{c-cylindrical resolution} if there exists a compact toric variety $V$, an affine toric variety $M_c$, and a torus-equivariant birational morphism  $\pi: M  \to V \times M_c$ which restricts to an isomorphism $\pi: M \backslash E \to  V \times \left( M_c \backslash \{o\} \right)$, where $\{o\} \in M_c$ is the unique torus fixed point, and $E = \pi^{-1}(0)$ is called the exceptional set. We set $M_0 =  V \times \left( M_0 \backslash \{ o\} \right) $. 
\end{definition}

In the picture above, this is equivalent to the condition that the set $P_V = \{ x \in \t_V^* \: | \: \langle \nu_i , \, x \rangle \geq -a_i, \, i \in I_1 \}$ is a polytope and that, outside of a neighborhood of the set $F_E = \cup_{i \in I_2} F_{\nu_i}$, $P$ coincides with a translate of the product polyhedron 
\begin{equation}\label{productpoly}
	P_{V \times M_c} = \{ (x_1, x_2) \in \t_V^* \oplus \t_C^* \: | \: x_1 \in P_V,  \, x_2 \in C  \}. 
\end{equation}
This is not always satisfied, for example take $M$ to be the blowup of $\C^2 \times \P^1$ along an axis of $\C^2$. Note that we necessarily have that $V \cong M_{P_V}$ and $M_c = M_C$, where $C = C(P)$.

\begin{definition}\label{asymptoticallyproduct}
	We say that an AK metric $\omega$ on $M = M_P$ is \emph{asymptotically c-cylindrical} if there exists an AK metric $\omega_V$ on $V$ and a toric K\"ahler cone metric $\omega_c$ on $M_c$ with radial function $r$ such that 
	\begin{equation}\label{asymptoticsproduct}
		\left| \pi_*g- g_0 \right|_{g_0}   < Cr^{- 2}, \hspace{.3in} \left|\nabla^{g_0} (\pi_*g- g_0) \right|_{g_0} < C r^{-1},
	\end{equation}
where $\omega_0 = \omega_V + \omega_c$ is the product metric and $\pi: M_P \backslash E \to  V \times M_0$ is the map of Definition \ref{asymptotic-complexstructure}, together with a technical condition on the associated symplectic structures. Specifically, we assume that there exists constants $a, C, C(\delta)>0$ such that a symplectic potential $u \in C^0(\overline{P}) \cap C^\infty(P)$ satisfies 
\begin{equation}\label{accyl-symplecticgrowth}
    |u(x)|  \leq C\left(|x|^a +1 \right) \hspace{.3in} \textnormal{for } x \in \overline{P},
\end{equation}
and 
\begin{equation}\label{accyl-symplecticgrowth2}
    \sum_{|\alpha| \leq 2}\left|\p_{\alpha}u \right| \leq C(\delta)\left(|x|^a +1 \right) \hspace{.3in} \textnormal{for } x \in P_\delta,
\end{equation}
for all $\delta << 1$, and moreover that there exists a $\bar{\delta} > 0$ such that 
\begin{equation}\label{accyl-symplecticboundary}
    \sup_{x \in \overline{P}\backslash P_{\bar{\delta}}} \left\{ \sum_{i,j,k = 1}^n \frac{1}{|x|^a + 1} \left|\frac{\p H_{ij}}{\p x^k } \right|^2 + \sum_{i,j,k,l = 1}^n \left| \frac{\p^2 H_{ij}}{\p x^k \p x^l } \right|^2 \right\} < C, \hspace{.3in} \H \geq C^{-1}\Id \textnormal{ on } P_{\bar{\delta}}.
\end{equation}
\end{definition}

In this case, we let $J_0, \omega_0, \mu_0$ be the corresponding product complex structure, K\"ahler form, and associated moment map on $ V \times M_c$, respectively. We denote by $\H_0, \G_0$ the data \eqref{Hdef}, \eqref{Gdef} associated to $\omega_0$.

\begin{remark}
   The technical conditions \eqref{accyl-symplecticgrowth}, \eqref{accyl-symplecticgrowth2}, \eqref{accyl-symplecticboundary} are all satisfied by the model metric $g_0$ itself by Lemma \ref{MSY-structure}. The conditions \eqref{accyl-symplecticgrowth}, \eqref{accyl-symplecticgrowth2} can be thought of as follows. The asymptotics \eqref{asymptoticsproduct} in particular give a bound of the form $|| \G - \G_0|| < C$. Heuristically at least, this gives a relationship between the hessians $\textnormal{Hess}(u)$ and $\textnormal{Hess}(u_0)$. Since $P$ and $P_V \times C$ coincide outside of a compact set up to a fixed translation, we can loosely imagine that the symplectic potential $u_0$ differs from the Guillemin potential $u_P$ by $O(|x|\log|x|)$, using the cone property. Therefore one would think that $u$ cannot grow faster than, say, $|x|^3$ on $P$, and similarly for the first and second derivatives on $P_\delta$. In practice, however, a comparison of this form involves rather complicated expressions in terms of $\mu_{0} \circ \mu_\omega^{-1}: P \to P_V \times C$, which in general are difficult to manage. We will see examples where this line of reasoning can be made to hold rigorously in Section \ref{section-chili} in the one-dimensional case (i.e. $P = [a, \infty)$), made significantly easier by the fact that the boundary $\p P$ is compact.
\end{remark}

\begin{lemma}\label{approx-momentmap}
	Let $\omega$ be any asymptotically c-cylindrical AK metric on $M$ with respect to $\omega_0 =  \omega_V + \omega_c$ on $ M_0 $.  Then using $\pi$ to identify $M \backslash E $ with $M_0$, we have 
		\begin{equation*}
	 			||\mu_\omega - \mu_{0}||_\t < C \big( \log(r) + 1 \big) .
	\end{equation*}
\end{lemma}
\begin{proof}
	Fix $Y$ such that $JY \in \t$, $f_Y = \langle \mu_\omega, \, JY \rangle$, $f^0_Y = \langle \mu_{0}, \, JY \rangle$. 
	\begin{equation*}
	\begin{split}
		|\nabla^{g_0}(\langle \mu - \mu_0 , \, Y\rangle)|_{g_0}  &= |\nabla^{g_0}(f_Y-f^0_Y)|_{g_0}  =  |d(f_Y-f^0_Y)|_{g_0} \\
			 &= |(\omega-\omega_0)(Y,-)|_{g_0} \leq |\omega-\omega_0| _{g_0} |Y|_{g_0} \leq C r^{-1}
	\end{split}
\end{equation*}
Suppose that $p \in M_0$ is any point and set $r_0 = r(p)$. Choose a fixed compact set $K \subset V \times M_c$ containing  $(V \times \{0\})$, and set $K_0 = K \backslash(V \times \{0\}) \subset M_0$. Then we have
\begin{equation}\label{momentmap-accyl-compactcompare}
\sup_{ K_0} || \mu_\omega - \mu_{0} || < C(K).
\end{equation}
 Choose any point $q \in K_0$. We define a piecewise smooth curve $\gamma:[0,d_2] \to M_0$ joining $q$ to $p$, as follows. Let $r_0 = r(p)$, so that $p$ lies in the hypersurface $\{r_0\} \times L \times V \subset \R_+ \times L \times V \cong M_0$. Set $\gamma_1: [0, d_1] \to M_0$ to be the unit-speed radial geodesic joining $q$ to the hypersurface $\{r_0\} \times L \times V$, and then $\gamma_2:[d_1, d_2] \to M_0$ be any path in $L \times V$ joining $\gamma_1(d_1)$ to $p$ of length less than the diameter $\textnormal{diam}(g_{L \times V})$, under the obvious identifications. Then we let $\gamma$ be the amalgamation of $\gamma_1$ and $\gamma_2$. Thus
\begin{equation*}
\begin{split}
	|(f_Y - f^0_Y)(p) - (f_Y - f^0_Y)(q) | &\leq \int_0^{d_1} \left|\ddt (f_Y - f^0_Y)(\gamma_1(s))\right| ds + \int_{d_1}^{d_2} \left|\ddt (f_Y - f^0_Y)(\gamma_1(s))\right| ds \\
  &\leq \int_0^{d_1} |\nabla^{g_0}(f_Y - f^0_Y)|_{g_0} ds + \int_{d_1}^{d_2} |\dot\gamma_2 |_{g_0}|\nabla^{g_0}(f_Y - f^0_Y)|_{g_0} ds  \\
  & \leq C \int_0^{d_1} r^{-1}(\gamma_1(s)) ds + C r_0^{-1}L_{g_0}(\gamma_2) \\
  & \leq C \big( \log(r_0) - \log(r(q)) + 1 \big) \leq C \big( \log(r_0) + 1 \big),
\end{split}
\end{equation*}
since the length $L_{g_0}(\gamma_2)$ satisfies $L_{g_0}(\gamma_2) = r_0 L_{g_{L \times V}}(\gamma_2)$. This together with \eqref{momentmap-accyl-compactcompare} gives the result. 
 \end{proof}
 
 Let $K_c = J_c\left( \rad \right) \in \t_C$ be the Reeb vector field on $(M_c, \omega_c)$.  Using $\pi$, we identify $\t = \t_V \oplus \t_C = \t_V \oplus \t_L$, recalling that we have a natural identification $\t_C \cong \t_L$. In particular, every vector field $Y \in \t$ is tangent to $L \times V$.  Moreover, we see that there is a unique global vector field $X_0$ on $M$ such that $\pi_*(JX_0) = K_c$, namely the one determined by $X_0 = -J\left( 0 \oplus K_c\right)$. 
 
 \begin{lemma}\label{Handr-lemma}
Let $\omega$ be any asymptotically c-cylindrical AK metric on $M$.  Then for any $p \in M$,
\begin{equation}\label{Handr}
	|| \mathbf{H} (\mu_\omega(p))||_{\t}^2 < C (r^2(p) + 1), 
\end{equation}
where $||\, \cdot \, ||_{\t}$ denotes any choice of intrinsic flat metric on $\t$. 
\end{lemma}

\begin{proof}
The product metric $g_0$ satisfies 
\begin{equation*}
	\begin{array}{lcl}
			g_0(Y_i, Y_j) = r^2 g_L(Y_i, Y_j), &  & Y_i, Y_j \in \t_L \\
			g_0(Y_i, Y_j) =  g_V(Y_i, Y_j), &  & Y_i, Y_j \in \t_V \\ 
			g_0(Y_i, Y_j) = 0, & & Y_i \in \t_C, Y_j \in \t_V,
	\end{array}
\end{equation*}
from which it follows immediately that 
\begin{equation*}
	|| \H_0||_{\t}^2  < C r^2. 
\end{equation*}
Then we have that
\begin{equation*}
\begin{split}
 |H_{ij}(\mu_\omega(p))| &= \left| \pi_*g(Y_i, Y_j) \right|(p) \\
        &\leq \left| (\pi_*g - g_0)(Y_i, Y_j) \right| + \left| g_0(Y_i, Y_j) \right|  \\
  &\leq C r^{-2} |Y_i|_{g_0}|Y_j|_{g_0} + Cr^2 \\
  & \leq C(r^2(p) + 1). 
 \end{split}
\end{equation*}
\end{proof}
 
 \subsection{Asymptotically c-cylindrical $(v, w)$-cscK metrics}

We are now in a position to prove Theorem \ref{mtheorem-asymptotics}. The precise statement is:
 \begin{theorem}\label{asym-product-satisfiesconditions}
 	Let $(M,\omega)$ be an asymptotically c-cylindrical $(v,\, w)$-cscK metric with weights $(v,\, w)$ satisfying \ref{exponential-decay}. Then $\omega \in \scaryH$ for  any $\varepsilon \in (0, \frac{1}{2})$, and consequently $P$ is weighted K-stable and $\beta$-uniformly K-stable for all $\beta \in (0,\beta^*)$.
 \end{theorem}
 
 \begin{proof}
	To verify condition \ref{asymptotics1}, we first observe that the function $r^2$ on $M_0$ is the hamiltonian potential for the vector field $K_c$. By \cite{MSY}, we know that $K_c$ is determined by a $b_C \in \t_C$ which lies in the dual cone $C^*$.  Translating this back to $M$, it follows that the vector field $X_0$ which satisfies $\pi_{*}JX_0 = 0 \oplus K_c$ is determined by $b_{X_0} = (0, \, b_C)$ and that $b_{X_0}$ lies in the interior of the dual recession cone $C(P)^*$.  It's easy to see from the definition that this implies that the linear function $\langle x, \, b_{X_0} \rangle$ is comparable to $||x||_{\t}$ on $P$.  By Lemma \ref{approx-momentmap} we know that the hamiltonian potential $f_{X_0}=\langle \mu_\omega, \, b_{X_0} \rangle$ is comparable to $r^2$ on $M \backslash E$, in the sense that 
    \begin{equation}\label{normcompare-withlog}
        \left| f_{X_0} - \frac{r^2}{2} \right| \leq C\big( \log(r) + 1 \big).
    \end{equation}
    In particular, this gives us 
    \[ \frac{f_{X_0}}{\log(r)} \geq \frac{r^2}{2\log(r)} - C, \]
    which is unbounded as $r \to \infty$. Therefore there exists a constant $C$ with 
    \begin{equation*}
        C(f_{X_0} + 1) \geq \log(r). 
    \end{equation*}
    Putting this together with \eqref{normcompare-withlog} and the previous discussion, we get that 
    \begin{equation}
        r^2 \leq C \log(r) + 2 f_{X_0} \leq C(f_{X_0} + 1) \leq C(||\mu_\omega||_{\t} + 1).
    \end{equation}
    So by Lemma \ref{Handr-lemma} we have that, as a function on $P$,
	\begin{equation}\label{Handnorm}
		|| \mathbf{H} ||_{\t}^2 < C \big( ||x||_{\t} + 1 \big), 
	\end{equation}
	 which immediately gives condition \ref{asymptotics1} of Assumption \ref{asymptotic-metric-general} for any $\varepsilon > 0$ as long as $v$ is exponentially decaying.

Condition \ref{asymptotics3} is immediate from \eqref{accyl-symplecticgrowth} and \eqref{accyl-symplecticgrowth2}, and condition \ref{asymptotics4} from \eqref{accyl-symplecticboundary}, so it remains only to show \ref{asymptotics2}. We have that $\H_0$ satisfies $\H_0 = \mu_{0}^*(\H_V + \H_C)$,  where $\H_V \in C^\infty(P_V, \t_V^* \otimes \t_V^*)$, $\H_C \in C^\infty(P_C, \t_C^* \otimes \t_C^*)$, extended to $P_{V \times M_c}$ in the obvious way.  Moreover, by \cite{MSY}, we know that $\H_C$ is homogeneous of degree $1$ on $P_C$.  Consequently we see that $	|| d(\H_V + \H_C)||_{\t} $ is bounded on $P_{V \times M_c}$.  Hence if we take any $p$ lying in the dense orbit of $M_0$, we see from Lemma \ref{gradexpressionlemma} that, for any $i,j$, 
\begin{equation}\label{grad-of-H0}
 	| \nabla^{g_0} (H_0)_{ij}|^2_{g_0}(p) = (H_0)_{kl}(H_0)_{ij,l}(H_0)_{ij,k} \leq C ||\mu_0(p)||_{\t_C} \leq C r^2.
\end{equation}
Next we observe that 
\begin{equation*}
\begin{split}
    \left|d(g-g_0)(Y_i, \, Y_j) \right|_{g_0} &\leq \left|\nabla^{g_0}g (Y_i, Y_j) \right| + \left| (g - g_0)(\nabla^{g_0}Y_i, Y_j)\right| + \left| (g - g_0)(Y_i, \nabla^{g_0}Y_j)\right|  \\
    & \leq |\nabla^{g_0}g| \, |Y_i| \, |Y_j| + \left|(g - g_0)\right| \big( |\nabla^{g_0}Y_i| \, |Y_j| + |Y_i| \, |\nabla^{g_0}Y_j|  \big), 
\end{split}
\end{equation*}
where $Y_i$ is our fixed basis of $\t$, and we have dropped the notation $\pi_*g$ and consider $g$ as a metric away from the apex on $V \times M_c$, and all norms are taken with respect to $g_0$. Arguing in local coordinates $(y_1, \dots, y_n)$ where $y_1 = r$ and $(y_2, \dots, y_n)$ are local coordinates on $L \times V$,  one can compute directly from the Christoffel symbols of $g_0$ that $\left| \nabla^{g_0}Y_i \right|_{g_0} < C r^2$.  Hence by \eqref{asymptoticsproduct} we have
\begin{equation}\label{d-of-difference-metrics}
     \left|d(g-g_0)(Y_i, \, Y_j) \right|_{g_0} \leq C r. 
\end{equation}
It follows that for each $i, j$, and for any $p \in V \times M_c$,
\begin{equation*}
\begin{split}
    \left| \nabla^{g} H_{ij} \right|_{g}\!(p) &= \left| d H_{ij} \right|_{g}\!(p) \leq C \left| d H_{ij} \right|_{g_0}\!\!(p)  \\
            &\leq C \left| d g_0(Y_i, Y_j) \right|_{g_0}\!\!(p)   + C \left| d (g -g_0)(Y_i, Y_j) \right|_{g_0}\!\!(p) \\
            & = C | \nabla^{g_0} (H_0)_{ij}|_{g_0}\!(p) + C \left| d (g -g_0)(Y_i, Y_j) \right|_{g_0}\!\!(p) \leq C r(p), 
\end{split}
\end{equation*}
by \eqref{grad-of-H0} and \eqref{d-of-difference-metrics}. On the other hand we have by Lemma \ref{gradexpressionlemma} again
\begin{equation*}
     \left| \nabla^{g} H_{ij} \right|_{g}^2(p) = H_{kl} H_{ij,k}H_{ij,l}\big|_{x = \mu_\omega(p)}.
\end{equation*}
Hence if $\bar{\delta}$ is as in Definition \ref{asymptoticsproduct}, if we take any point $p \in \pi(\mu_\omega^{-1} (P_{\bar{\delta}})) \subset V \times M_c$, we see that 
\begin{equation*}
\begin{split}
    \sum_{k = 1}^n |H_{ij,k}|^2(\mu_\omega(p)) &\leq C\left| \nabla^{g} H_{ij} \right|_{g}^2(p) \leq Cr^2(p) \\
            &\leq C(f_{X_0}(\mu_\omega(p)) + 1) \leq C(||\mu_\omega(p)||_{\t} + 1).
\end{split}
\end{equation*}
In other words, $||d \H ||_{\t}^2 \leq C(||x||_{\t} + 1)$ on $P_{\bar{\delta}}$, where everything is measured with respect to the Euclidean norm. Together with the assumed bound $||d\H||_{\t} < C(|x|^a + 1)$ on $\overline{P}\backslash P_{\bar{\delta}}$, this gives condition \ref{asymptotics2} of Definition \ref{asymptotic-metric-general} for any $\varepsilon > 0$. 
\end{proof}

 \section{Semisimple principal fibrations}\label{section-ssfibrations}

 \subsection{Definitions}
We can apply the results of the previous sections to study certain non-toric manifolds as follows, using the \emph{semisimiple principal fibration} construction \cite{ACGT-overcurve, ApJuLa,LahdiliWeighted,simonYTD}. We begin with a K\"ahler manifold $(M, \omega)$ with $\dim_\C M = k$ together with an $\omega$-hamiltonian action of a $k$-dimensional real torus $\T$. As before we denote $\t = \textnormal{Lie}(\T)$ and $\Gamma \subset \t$ the associated lattice. Let $B = B_{1} \times \dots \times B_\ell$ be a compact K\"ahler manifold endowed with a product metric $\omega_B = \sum_{a=1}^\ell \omega_a$. We assume that $\omega_a$ are individually cscK and that $[\omega_{B_a}] \in H^2(B_a, 2\pi \Z)$ are integral. As such one can ask for a principal $\T$-bundle $\pi_B:U \to B$ and a connection form $\theta \in \Omega^1(U,\t)$ whose curvature satisfies 
\[ d\theta = \sum_{a=1}^\ell \pi_B^*\omega_{B_a} \!\otimes p_a, \hspace{.2in} p_a \in \Gamma. \]
Then we define, at the smooth level, the \emph{semisimple fibration over $B$ with fiber $M$} to be the associated bundle 
\[ Y = M \times_{\T} U, \hspace{.2in} \pi_B:Y \to B, \]
which by construction comes equipped with an effective $\T$-action. Here we use the same notation $\pi_B$ to denote the projection induced by $\pi_B: U \to B$. In fact, one can verify that the natural CR-structure $(\mathcal{D},J_B)$ on $U$ induced by $B$ gives rise to an associated integrable complex structure $J_Y$ on $Y$ \cite{ApJuLa}.

Before moving on to discuss the metric aspects of the construction, for simplicity we first specialize to the case of semisimple principal \emph{toric} fibrations (semisimple rigid toric fibrations in the language of \cite{ACGT-overcurve}). In this situation, we suppose that $M$ is toric in the sense of Definition \ref{noncompacttoric},  $\omega$ is an AK metric on $M$, and $U \to B$ is a $\T$-bundle, where $\T$ is now the real torus underlying the given $(\Cstar)^k$-action on $M$. In this setting we have a moment map $\mu_\omega: M \to \t^*$ whose image $\mu_\omega(M) = \overline{P}$ is a Delzant polyhedron fixed up to translation by the cohomology class $[\omega] \in H^2(M,\R)$. As usual we fix some representative for $P$ in its translation class. 

Given this setup, there is also an associated K\"ahler metric $\omega_Y$ on $Y$. This is defined on the product $M \times U$ by 
\begin{equation}\label{bundlecompatible}
\begin{split}
    \omega_Y &= \omega + \sum_{a=1}^\ell c_a \pi_B^*\omega_{B_a} + d(\langle \mu_\omega, \theta \rangle) \\
            & = \omega + \sum_{a=1}^\ell (\langle \mu_\omega, p_a \rangle + c_a) \pi_B^*\omega_{B_a} + \langle d\mu_{\omega} \wedge \theta \rangle,
\end{split}
\end{equation}
where $c_1, \dots, c_\ell$ are chosen such that the affine-linear function $\langle x, \, p_a \rangle + c_a > 0$ on $P$. Here we use the notation from \cite{ApJuLa}, and $\langle \cdot , \, \cdot  \rangle$ denotes the dual pairing between $\t$ and $\t^*$. Note that in the non-compact setting the existence of such $c_a$ is a non-trivial assumption, and indeed implies that $p_a \in C(P)^*.$ Then $\omega_Y$ defined by \eqref{bundlecompatible} descends to a well-defined K\"ahler metric on $Y$ \cite{ApJuLa}. Moreover, the induced $\T$-action on $Y$ is $\omega_Y$-hamiltonian. In fact, if we pull back the moment map $\mu_\omega$ to a function $\mu_\omega: M \times U \to \t^*$, then this descends as well to a function $\mu_{Y}:Y \to \t^*,$ which is precisely the moment map for the $\T$-action on $Y$ with respect to $\omega_Y$. In particular, the image of $\mu_{Y}$ is precisely $P$.

The following result is a local (albeit arduous) computation, which thus holds equally in the current non-compact setting:
\begin{lemma}[{\cite[Lemma 5.9]{ApJuLa}}]\label{fibrationweights}
    Set $s_a = \Scal(\omega_{B_a})$ which is constant by assumption, and set $n_a = \dim_{\C}B_a$. Define positive rational functions $p,q$ on $P$ by  
    \[ p(x) = \Pi_{a=1}^\ell (\langle p_a ,  x \rangle + c_a)^{n_a}, \hspace{.3in} q(x) = \sum_{a = 1}^\ell \frac{s_a}{\langle p_a ,  x \rangle + c_a}.\]
    Let $\omega$ be an AK metric on $M$ and $\omega_Y$ the associated bundle-compatible metric on $Y$. Then $\omega_Y$ is a $(v,\,w)$-cscK metric on $Y$ if and only if $\omega$ is a $(\tilde{v}, \tilde{w})$-cscK metric on $M$, where 
    \[\tilde{w}(x) = p(x)(w(x) - q(x)v(x)). \]
\end{lemma}
In particular, if $(v,\, w)\in \mathcal{W}$ are exponentially decaying weights on $P$, then after perhaps shrinking $\beta^*$ we will have that $(\tilde{v}, \, \tilde{w}) \in \mathcal{W}$ as well. In this way, the semisimple principal fibration setup can be viewed as a form of dimensional reduction of the problem, where the extra structure is encoded in the new weights.  

 \subsection{K\"ahler-Ricci solitons}\label{section-KRS}

Let $Y$ be a complex manifold. A \emph{shrinking K\"ahler-Ricci soliton} on $Y$ is a pair $(\omega, \, X)$ where $\omega$ is a K\"ahler metric on $M$and $X$ is a holomorphic vector field, which solve the equation 
\begin{equation}\label{SKRS}
    \Ric_{\omega} + \frac{1}{2}\mathcal{L}_X\omega = \omega.
\end{equation}
We say that $(\omega, \, X)$ is \emph{gradient} if $X = \nabla^g f$ for some $f \in C^\infty(Y)$, where $g$ is the associated Riemannian metric to $\omega$. In the non-compact setting, complete shrinking K\"ahler-Ricci solitons are of interest with respect to the K\"ahler-Ricci flow \cite{EMT,Naber, BCCD}, whereas in the compact setting they serve as a natural generalization of positive K\"ahler-Einstein metrics, see for example \cite{BWN-optimal,DSz}. 

The soliton vector field $X$ of a gradient K\"ahler-Ricci soliton always has the property that $JX$ is Killing, and hence generates an action of a finite-dimensional real torus $\T$ on $Y$. In fact, at least when the Ricci curvature of $\omega$ is bounded, the action always complexifies to an effective and holomorphic action of $(\Cstar)^k$ on $Y$ (\cite[Theorem 5.1]{ConDerSun}, \cite[Theorem 4.2]{uniqueness}). In this way, shrinking gradient K\"ahler-Ricci solitons are a natural starting place for the study of complete toric or ``partially'' toric geometry, as we will see below. 

 \subsubsection{K\"ahler-Ricci solitons on toric manifolds}
 We begin with a brief discussion of the simplest case, where $Y = M$ is toric.
  \begin{lemma}\label{toricSGKRSweight}
 	Suppose that $\omega \in 2\pi c_1(M)$ is an AK metric on a toric manifold $M$ which is a $v$-soliton with weight function 
 	\begin{equation}
 		v(x) = e^{-\langle x, \, b_X \rangle},
 	\end{equation}
 	for $b_X \in \t$. Then $\omega$ is a shrinking gradient K\"ahler-Ricci soliton with respect to the fundamental vector field $X$ on $M$ associated to $b_X$. 
 \end{lemma}
 This is well-known, see \cite{BB, BWN-optimal, uniqueness, HanLi}, among others. Thus, any AK K\"ahler-Ricci soliton $\omega$ on a toric manifold $M$ is a $v$-soliton with exponentially decaying weights with derivatives as in Definition \ref{exponential-decay}.  If $\omega$ is asymptotically c-cylindrical, then we are precisely in the setting of this paper. This is the geometry that we see, at least conjecturally, for the known examples. See in particular \cite[Conjecture 1.2]{CCD1} concerning the shrinker on the blowup $\Bl_p(\C\P^1 \times \C)$ of $\C\P^1 \times \C$ at a point constructed in \cite{BCCD}. In practice, however, it is often difficult to verify any precise asymptotics. We will see below (Proposition \ref{chili-asymptotics}) that this is indeed the geometry that we get for certain $v$-solitons on $\C$ which give rise to complete shrinking K\"ahler-Ricci solitons on the total space of a root $L \to B$ of the canonical bundle over a compact K\"ahler-Einstein Fano manifold.

\subsubsection{K\"ahler-Ricci solitons on semisimple fibrations}
The toric picture above can be seen as a limiting case of the semisimple principal toric fibration construction with zero-dimensional base. Let $B = \Pi_{a} (B_a, \omega_a)$ be a product of positive K\"ahler-Einstein manifolds $(B_a, \omega_{B_a})$ of complex dimensions $n_a$, and that $ [\omega_a] = \frac{c_1(B_a)}{k_a} \in H^2(B_a, 2\pi \Z)$.  Let $U$ be the principal $\T$-bundle over $B$  with Chern class 
 \begin{equation}
 2\pi c_1(U)=\sum [\omega_a] \otimes p_a, \hspace{.3in} p_a \in \Gamma.
 \end{equation}
 Let $(M, \,\omega)$ be an anticanonically polarized toric $k$-dimensional manifold with $\omega$ an AK metric. In particular $M$ is endowed with an isometric and hamiltonian $\T$-action with moment map $\mu_\omega$ and moment image $P \subset \t^*$, which we assume to be canonically normalized.
As a corollary of Lemma \ref{fibrationweights}, we have the following from \cite{ApJuLa}:
 \begin{lemma}[{\cite[Lemma 5.11]{ApJuLa}}]\label{fibrationsolitons}   Suppose that each factor $p_a x + k_a$ is strictly positive on $P$, and that $\omega$ is a shrinking $v$-soliton for the weight function
\begin{equation}
	v(x) :=\Pi_{a=1}^\ell (\langle p_a ,  x \rangle + k_a)^{n_a} e^{-\langle b_X,\, x \rangle}.
\end{equation}
 Then the compatible metric $\omega_Y$ associated to $\omega$ by \eqref{bundlecompatible} on 
\begin{equation}
	Y = M \times_{\T} U \to B 
\end{equation}
is a shrinking gradient K\"ahler-Ricci soliton with soliton vector field $ X_{b_X}$. 
 \end{lemma} 
The main difference between this and Lemma \ref{fibrationweights} is that here one must verify that the bundle compatible metric $\omega_Y$ associated to $\omega$ has the property that $\omega_Y \in 2\pi c_1(Y)$.

 \subsubsection{K\"ahler-Ricci solitons on line bundles}\label{section-chili}

In this section we discuss the examples of non-compact complete shrinking K\"ahler-Ricci solitons which are built out of certain negative line bundles $L \to B$ over a compact Fano manifold $B$. Starting with Cao \cite{Caosoliton} and Feldman-Ilmanen-Knopf \cite{FIK}  in the case of negative line bundles over $\C\P^{n-1}$, many authors have successfully used the Calabi Ansatz to furnish complete non-compact solitons. Notably there is the work of Li \cite{ChiLiexamples}, extending the construction to direct sums of line bundles over $B$, and those of Futaki, Futaki-Wang \cite{Fut, FutWang} in the line bundle case but allowing for certain toric $B$ which are not necessarily K\"ahler-Einstein. In this section we will focus on the approach of \cite{ChiLiexamples}, and for simplicity we have attempted to keep more or less the same notation conventions that appear in \cite{ChiLiexamples}. 

Let $B$ be a $d$-dimensional K\"ahler-Einstein Fano manifold with $\Ric_{KE} = \tau \omega_{KE}$. We let $L \to B$ be a negative line bundle with a hermitian metric $h$ with the property that the curvature $c_1(L,h) = - \kappa \omega_{KE}$ (i.e. $L$ is a $\Q$-multiple of $K_B$). The main result for our purposes here is 
\begin{theorem}[{\cite{ChiLiexamples}}]\label{chiliexistence}
    For appropriate choices of $\tau, \kappa$, there exist complete shrinking gradient K\"ahler-Ricci solitons $\tilde{\omega}$ on the total space $Y$ of the direct sum bundle 
    \[ Y = L^{\oplus k} = L \oplus \dots \oplus L \to B, \]
    constructed via the Calabi Ansatz \eqref{calabiansatz}
    \[ \tilde{\omega} =  \pi_B^*\omega_{B} + i\p\bp P(s),\]
    where $s = ||\cdot ||_{h}^2$ is the induced norm function on $Y$ by $h$. Furthermore, the soliton vector field $X$ generates the $\R_+$-action on $Y$ given by multiplication by positive real scalars on the fibers. 
\end{theorem}
A consequence of \cite[Lemma 5.5]{ApJuLa} is:

 \begin{lemma}\label{CalabiAnsatzidentificationlemma}
     Let $U_k \to B$ be the $\T = U(1)^k$-bundle associated to $Y$. Then any metric $\omega_Y$ of the form $\omega_Y = \pi_B^*\omega_{B} + i\p\bp P(s)$ on $Y$ is bundle compatible with respect to the fibration structure $Y = \C^k \times_{\T} U_k$, and moreover the associated metric $\omega$ on $\C^k$ is given by $\omega = i\p\bp P(|z|^2)$.
 \end{lemma}

 The shrinker $\tilde{\omega} = \pi_B^{*}\omega_{KE} + i \p \bp P(s)$ on $Y$ given by Theorem \ref{chiliexistence} is determined in terms of some data that we will now very briefly summarize. For the full details of the construction, see \cite{ChiLiexamples}. Let $r = \log(s)$, so that $\p_r = s\p_s$. Then define 
 \begin{equation*}
    \begin{array}{ccc}
      \phi(r) = P_r(r) = sP'(s), &   & F(r) = \phi_r(r) = s(P''(s) + sP'(s)). 
    \end{array} 
 \end{equation*}
Then the function $\phi$ is monotone, so that $r = r(\phi)$, and we can write $F(\phi) = F(r(\phi))$. Moreover, $F \geq 0$ and $\phi \geq 0$. The metric we seek is then specified by solving an equation for $F$ whose solution in our setting is precisely:
 \begin{equation} \label{ChiLiF}
     F(\phi) = (1 + \kappa \phi)^{-d}\phi^{1-k}\left(\sum_{j=0}^{d+k} \frac{h^{(j)}(\phi)}{\mu^{j+1}}\right),
 \end{equation}
 where 
 \begin{equation*}
     h(\phi) = \tau(1+\kappa \phi)^d\phi^k - k(1 + \kappa \phi)^{d+1}\phi^{k-1}.
 \end{equation*}
 Fix a finite value $\phi_0$ of $\phi$ on $\C$, with corresponding values $r(\phi_0) = r_0$, $s(r_0) = e^{r_0} = s_0$. From \cite[Section 5.2]{ChiLiexamples}, we get 
 \begin{equation}\label{ChiLiPprime}
     \phi = \phi_0 s_0^{-p}s^{p} e^{-G_{\phi_0}(\phi)},
 \end{equation}
 where 
 \begin{equation}\label{ChiLiG}
     G_{\phi_0}(\phi) = \int_{\phi_0}^\phi \frac{pu - F(u)}{puF(u)} du
 \end{equation}
has the property that 
\begin{equation*}
    \lim_{\phi \to \infty} G_{\phi_0}(\phi) = C_0 
\end{equation*}
exists and is finite. Using this, Li shows in \cite[Equation (26)]{ChiLiexamples} that an after an appropriate scaling $D_0$, we have $||g - D_0g_p||_{g_0} \to 0$ at infinity using the flow of $X$. In what follows, we will show that, at least in the case when $k = 1$, this can be improved to to \eqref{asymptoticsproduct}. To this end, set $p = \tau - k \kappa$. Then we compute
\begin{equation}\label{FisOone}
\begin{split} 
F(\phi) &= (1 + \kappa \phi)^{-d}\phi^{1-k}\left(\sum_{j=0}^{d+k} \frac{h^{(j)}(\phi)}{\mu^{j+1}}\right) = (1 + \kappa \phi)^{-d}\phi^{1-k}\frac{h(\phi)}{\mu} + O(1) \\
    & = \frac{(\tau - k \kappa)}{\mu} \phi + O(1) = p \phi + O(1).
\end{split}
\end{equation}
Hence, for all all $u$ sufficiently large we have
\[ \left| \frac{pu - F(u)}{puF(u)}\right| \leq \frac{C_1}{(pu)^2 - C_1(pu)}.\]
Integrating gives us 
\begin{equation}\label{GtoGinfinitycompare}
    \left| G_{\phi_0}(\phi) - G_{\phi_0}(\infty) \right| \leq \int_{\phi}^\infty \frac{C_1}{(pu)^2 - C_1(pu)} du = \frac{1}{p}\log\left( \frac{p \phi - C}{p\phi}\right).
\end{equation}
Set 
\begin{equation*}
    D_0 = \frac{p}{2p-1} \phi_0 s_0^{-p}e^{-G_{\phi_0}(\infty)}.
\end{equation*}

 The main goal of this section will be to prove the following:
 \begin{prop}\label{chili-asymptotics}
     Let $Y$ be as in Theorem \ref{chiliexistence} with $k = 1$, i.e. $Y$ is the total space of a root $L \to B$ of the canonical bundle $K_B$. Then the $v$-soliton metric $\omega$ on $\C$ associated to the complete shrinker on $Y$ via Lemma \ref{CalabiAnsatzidentificationlemma} is asymptotic to the conical metric 
     \[ \omega_p = i D_0 \p \bp s^p, \]
     where $s = |z|^2$, in the sense of Definition \ref{asymptoticallyproduct}. 
 \end{prop}

 \begin{remark}
     We already know from Proposition \ref{mainprop2} that the metric $\omega$ above lies in $\scaryH$ for any $\varepsilon > 0$. Indeed by Lemmas \ref{fibrationsolitons}, \ref{vsolitonw}, $\omega$ is a $(\tilde{v}, \, \tilde{w})$-cscK metric where both $\tilde{v}, \tilde{w}$ are given by a polynomial times $e^{- \lambda x}$ for some $\lambda > 0$.
 \end{remark}

\begin{proof}
 
Recall that our metric on $\C$ is given by 
\begin{equation*}
\begin{split}
    \omega = i\p\bp P(s) &= \left(P'(s) + sP''(s)\right) idz \wedge d\bar{z} \\
        & = s^{-1}P_{rr}idz \wedge d\bar{z} \\
        &= s^{-1}F idz \wedge d\bar{z} = (ps^{-1}\phi + O(s^{-1})) idz \wedge d\bar{z}
\end{split}
\end{equation*}
Then we have by \eqref{ChiLiPprime} that
\begin{equation*}
\begin{split}
    \left| g - D_0g_p \right|_{g_p}^2 &= \left(\frac{ps^{-(p-1)}}{2p -1}\right)^2 \left[ s^{-1} \phi - (2p-1)D_0 s^{p-1} + O(s^{-1})\right]^2 \\
    &\leq \left(\frac{ps^{-(p-1)}}{2p -1}\right)^2\left[\phi_0s_0^{-p}s^{(p-1)}\left( e^{-G_{\phi_0}(\phi)} - e^{-G_{\phi_0}(\infty)}\right)  + O(s^{-1})\right]^2
\end{split}
\end{equation*}
Using that $e^{-x}$ is uniformly 1-Lipshitz for $x \geq 0$, it follows that $|e^{-G_{\phi_0}(\phi)} - e^{-G_{\phi_0}(\infty)}| \leq |G_{\phi_0}(\phi)  -G_{\phi_0}(\infty)|$ as long as $\phi$ is sufficiently large. Therefore, by \eqref{GtoGinfinitycompare}, for all $\phi$ sufficiently large we have
\begin{equation*}
\begin{split}
    \left| g - D_0g_p \right|_{g_p}^2 &\leq D_0^2\left| G_{\phi_0}(\phi) - G_{\phi_0}(\infty) \right|^2 + \left| G_{\phi_0}(\phi) - G_{\phi_0}(\infty) \right| O(s^{-p})  + O(s^{-2p}) \\
     &\leq \frac{D_0^2}{p^2} \left|\log\left(1 - \frac{C}{p\phi}\right) \right|^2 + \left|\log\left(1 - \frac{C}{p\phi}\right) \right| O(s^{-p}) + O(s^{-2p}) \\
     &\leq \left(\frac{CD_0}{p^2}\right)^2 \phi^{-2} + \frac{CD_0}{p^2} \phi^{-1}O(s^{-p}) + O(s^{-2p}) \leq Cs^{-2p},
\end{split}
\end{equation*}
using \eqref{ChiLiPprime} and the fact that $|e^{-G_{\phi_0}(\phi)} -e^{-G_{\phi_0}(\infty)}| < C$. Since the distance function $\rho_0$ with respect to $g_p$ behaves like $\rho_0 \sim s^{\frac{p}{2}}$, it follows that
\begin{equation*}
\begin{split}
    \left| g - D_0g_p \right|_{g_p} \leq C s^{-p} \leq C \rho_0^{-2},
\end{split}
\end{equation*}
verifying the $C^0$ condition in \eqref{asymptoticallyproduct}. The Christoffel symbol $\Gamma^1_{11}$ for $g_p$ satisfies $\Gamma^1_{11} = g_p^{11}\p_z g_{p,11} = \frac{p-1}{z}$.  Hence we compute
\begin{equation}\label{ChiLigradg}
\begin{split}
    |\nabla^{g_p}g|_{g_p}^2 &= \left(\frac{s^{-(p-1)}}{2p -1}\right)^3\left[\p_z (s^{-1}F) - \frac{p-1}{z}(s^{-1}F) \right]^2\\
    &\leq \left(\frac{s^{-(p-1)}}{2p -1}\right)^3\left[s^{-1}\p_z F - \frac{p}{z}(s^{-1}F) \right]^2 \\
    & \leq \left(\frac{s^{-(p-1)}}{2p -1}\right)^3\left[s^{-1}|\p_\phi F||\p_{z}\phi| + \left| \frac{p}{z}(s^{-1}F)\right| \right]^2 \\
    &\leq \left(\frac{s^{-(p-1)}}{2p -1}\right)^3\left[Cs^{-1}|\p_{z}\phi| + \left| \frac{p}{z}(s^{-1}F)\right| \right]^2
\end{split}
\end{equation}
We have seen that $s^{-1}F = O(s^{p-1})$, so that $(zs)^{-1}F = O(s^{p - \frac{3}{2}})$. For the other term, we compute from \eqref{ChiLiPprime} that 
\begin{equation*}
    \p_z \phi = C \p_z (s^p e^{-G_{\phi_0}(\phi)}) = C p \bar{z}s^{p-1}e^{-G_{\phi_0}(\phi)} - \p_z G_{\phi_0}(\phi) s^p e^{-G_{\phi_0}(\phi)}.
\end{equation*}
Then we see that 
\[\p_z G_{\phi_0}(\phi) = \p_z \phi \left( \frac{p\phi - F(\phi)}{p\phi F(\phi)}\right),\]
which combining with the above gives 
\[\p_z\phi \left(1 -  s^p e^{-G_{\phi_0}(\phi)}\left( \frac{p\phi - F(\phi)}{p\phi F(\phi)}\right)\right) = C p \bar{z}s^{p-1}e^{-G_{\phi_0}(\phi)}. \]
Finally, we observe using \eqref{ChiLiF}, \eqref{ChiLiPprime} that 
\[s^p e^{-G_{\phi_0}(\phi)}\left( \frac{p\phi - F(\phi)}{p\phi F(\phi)}\right) = O(s^{-p}),\]
and hence 
\[|\p_z \phi| \leq C p \bar{z}s^{p-1}e^{-G_{\phi_0}(\phi)} \leq Cs^{p-\frac{1}{2}}. \]
Combining this with \eqref{ChiLigradg}, we get that 
\begin{equation*}
     |\nabla^{g_p}g|_{g_p}^2 \leq C \left(\frac{s^{-(p-1)}}{2p -1}\right)^3 s^{2p - 3} \leq C s^{-p},
\end{equation*}
which in particular gives us the first derivative estimate of \eqref{asymptoticallyproduct}. It is straightforward to verify \eqref{accyl-symplecticgrowth} and \eqref{accyl-symplecticgrowth2}, using the fact that in this setting the symplectic potential $u$ for $\omega$ satisfies $u'' = F^{-1}$. Note that, in the notation of the previous sections, the variable $\phi$ here corresponds to what we have typically named $x$.
\end{proof}

It seems likely that the same conclusion holds for arbitrary values of $k$, where $\omega_p = iD_0 \p \bp s^p$, where now $s = |z|^2$ on $\C^k$. 
\begin{conj}
    The metric $\omega$ on $\C^k$ determined by Lemma \ref{CalabiAnsatzidentificationlemma} is asymptotic to $\omega_p$ in the sense of Definition \ref{asymptoticallyproduct}.
\end{conj}

\bibliographystyle{abbrv}
\bibliography{references}

\end{document}